\documentclass[11pt]{article}
\usepackage[a4paper,margin=28mm]{geometry}
\usepackage[T1]{fontenc}
\usepackage{lmodern,amsmath,amssymb,amsthm,mathrsfs,microtype}
\usepackage{xcolor}
\usepackage{float}
\usepackage{enumitem}
\usepackage{pgfplots}
\pgfplotsset{compat=1.18}
\definecolor{hestonsuper}{HTML}{0072B2}
\definecolor{hestoncritical}{HTML}{D55E00}
\usepackage[colorlinks=true,linkcolor=blue!50!black,citecolor=blue!50!black,urlcolor=blue!50!black]{hyperref}
\allowdisplaybreaks[2]
\numberwithin{equation}{section}
\newtheorem{theorem}{Theorem}[section]
\newtheorem{proposition}[theorem]{Proposition}
\newtheorem{lemma}[theorem]{Lemma}
\newtheorem{corollary}[theorem]{Corollary}
\theoremstyle{definition}

\theoremstyle{remark}

\newcommand{\R}{\mathbb R}
\newcommand{\E}{\mathbb E}
\newcommand{\PP}{\mathbb P}
\newcommand{\HH}{\mathbb H}
\newcommand{\CC}{\mathcal C}
\newcommand{\Lop}{\mathscr L}
\newcommand{\loc}{\mathrm{loc}}

\title{Sharp asymptotics for a free boundary touching a degenerate
boundary: the perpetual Heston put}
\author{Takwon Kim\thanks{School of Mathematics, Statistics and Data
Science, Sungshin Women's University, 2 Bomun-ro 34da-gil, Seongbuk-gu,
Seoul 02844, Republic of Korea. Email: takwon@sungshin.ac.kr.
ORCID: 0000-0002-2595-5319.}
\and Junkee Jeon\thanks{Department of Applied Mathematics, Kyung Hee
University, 1732 Deogyeong-daero, Giheung-gu, Yongin 17104, Republic of
Korea. Email: junkeejeon@khu.ac.kr. ORCID: 0000-0002-1091-9745.}
\and Jehan Oh\thanks{Corresponding author. Department of Mathematics,
Kyungpook National University, 80 Daehak-ro, Buk-gu, Daegu 41566,
Republic of Korea. Email: jehan.oh@knu.ac.kr.
ORCID: 0000-0001-9602-5111.}}
\date{}

\begin{document}
\maketitle

\begin{abstract}
We study the exercise boundary of the perpetual American put in the
Heston model near zero variance. There the pricing operator
degenerates, and the exercise boundary, a free boundary in the
stock--variance half-plane, touches the boundary of the half-plane.
No Feller condition is imposed. The boundary is analytic at positive
variance and extends as a regular $C^1$ curve tangent to the
zero-variance axis. Its approach to the axis follows one of three
laws, according to whether the Feller ratio $2\kappa\theta/\sigma^2$
is larger than, equal to, or smaller than one: a logarithmic correction
to linear order, an exponential correction in the square root of the
logarithm, or a fractional power. In the first two cases the leading
coefficients are explicit and the tangent is not H\"older continuous
at the axis. In the third case the exponent is selected by a possibly
nonsymmetric stable operator on the axis, and the amplitude is finite
and positive.
\end{abstract}

\noindent\textbf{Keywords.} Free boundary, obstacle problem, degenerate
elliptic equation, perpetual American put, Heston model, stable
operator.

\medskip
\noindent\textbf{2020 Mathematics Subject Classification.} Primary
35R35, 35J70; Secondary 35J86, 35B40, 60G40, 91G20.

\section{Introduction}

This paper concerns a free boundary that touches a boundary on which the
elliptic operator degenerates. The problem comes from the perpetual
American put in the Heston model. In the stock--variance half-plane the
exercise boundary is the free boundary of an obstacle problem, the
diffusion coefficient vanishes on the zero-variance axis, and the
variance has a positive inward drift there. We address three
questions: whether the exercise boundary has a nontrivial limit at
zero variance, what regularity the boundary curve has at that endpoint,
and what asymptotic law governs its approach. The last question leads
to three distinct regimes, with different rates and leading
coefficients. In one of them the rate is selected by a nonlocal
equation on the axis.

Let $(S,Y)$ solve
\begin{equation}\label{eq:sde}
\begin{aligned}
 dS_t&=(r-q)S_t dt+S_t\sqrt{Y_t} dB_t^1,\\
 dY_t&=\kappa(\theta-Y_t) dt+\sigma\sqrt{Y_t} dB_t^2,
 \qquad  d\langle B^1,B^2\rangle_t=\rho dt.
\end{aligned}
\end{equation}
Throughout the paper,
\begin{equation}\label{eq:parameters}
 r>0,\qquad 0\le q\le r,\qquad
 \kappa,\theta,\sigma,K>0,\qquad |\rho|<1.
\end{equation}
At zero variance we use the usual nonnegative CIR solution with
immigration $\kappa\theta>0$. Zero is not absorbing, and we do not
impose the Feller condition $2\kappa\theta\ge\sigma^2$.
The value and exercise boundary are
\begin{equation}\label{eq:value}
 V(s,v)=\sup_\tau\E_{s,v}[e^{-r\tau}(K-S_\tau)^+],
 \qquad s>0,\quad v\ge0,
\end{equation}
and
\begin{equation}\label{eq:boundary}
 b(v)=\sup\bigl(\{s\in(0,K):V(s,v)=K-s\}\cup\{0\}\bigr).
\end{equation}
The stopping times are taken in the completed Brownian filtration.
The payoff at $\tau=\infty$ is zero. We write
\begin{equation}\label{eq:notation}
 a=\kappa\theta,\qquad D=\frac{\sigma^2}{2},\qquad
 \nu=\frac aD,\qquad B=b(0),\qquad d(v)=\log\frac{B}{b(v)}.
\end{equation}
Positivity of $b$ is proved below. The two constants that determine
the asymptotic regimes are
\begin{equation}\label{eq:driftparameters}
 \Lambda=\frac\sigma{\sqrt{1-\rho^2}},\qquad
 c=\frac{2(r-q)-\rho\nu\sigma}{\sigma\sqrt{1-\rho^2}}.
\end{equation}

Daskalopoulos and Feehan~\cite{DF} established weighted regularity
estimates for the elliptic Heston obstacle problem. Lamberton and
Terenzi~\cite{LT} proved continuity and variance monotonicity for
finite-maturity Heston put values without the Feller condition.
We use Lamberton and Terenzi's continuity and monotonicity results and
pass to the perpetual problem by a uniform discount estimate.
Nguyen and Chau~\cite{NC} study
finite-maturity value regularity and smooth fit under the Feller
condition, with positive initial variance. Chen, Jin, and
Song~\cite{CJS} studied American options in stochastic volatility models
with a negative Fichera function on the degenerate boundary. These
results concern value regularity and do not give the boundary
equivalents below.

The weighted derivative estimate in \cite[Section 5, (5.2)]{DF} and
continuity of the gradient at zero variance are different properties.
Appendix~\ref{sec:explicitcusp} gives a local example that satisfies
the hypotheses of \cite[Theorem 1.1]{DF} and has bounded weighted
derivatives, but whose gradient is not continuous at the contact point
on the axis. In this example the weighted bounds hold, whereas the
joint $C^1$ regularity up to the axis that is part of the space in
\cite[Definition 2.2]{DF} fails.

Shahgholian and Yeressian~\cite{SY} studied the obstacle problem
$\operatorname{div}(x_1^a\nabla u)=\chi_{\{u>0\}}$ with $a<1$ and zero
Dirichlet data on $\{x_1=0\}$. They proved that the free boundary
touches the fixed boundary tangentially and is a $C^1$ graph near the
contact point, with a uniform modulus. Let $\nu>1$ and $c=0$, and
remove the mixed derivative as in Section~\ref{sec:axis}. The
substitution $u=y^{\nu-1}w$ then turns the leading part
$Dy(\Delta w+\frac\nu yw_y)$ of our operator into
$D\operatorname{div}(y^{2-\nu}\nabla u)$. This is the operator of
\cite{SY} with $a=2-\nu$, and bounded $w$ give zero Dirichlet data for
$u$ on the axis. Our aim is to
determine the exact approach rate and its coefficient, including the
effect of the tangential drift, and to treat $0<\nu\le1$, where this
correspondence is not available. For free boundaries that meet a fixed
boundary in the classical obstacle problem, see~\cite{An,ShU} and the
monograph~\cite{PSU}.

For $0<\nu<1$ the resulting boundary equation is nonlocal. In
divergence form our operator carries the weight $y^\nu$, which is the
weight of the Caffarelli--Silvestre extension~\cite{CS} for the
fractional Laplacian of order $1-\nu$. Accordingly, the thin contact
layer acts on the axis through a Riesz potential of order $1-\nu$, and
the tangential drift makes the associated stable operator nonsymmetric.
We use the half-line classification of Dipierro, Ros-Oton, Serra, and
Valdinoci~\cite{DRSV} for such operators.

This contact problem differs from regularity at the intersection of
fixed degenerate and nondegenerate boundary portions. It also requires
more information than a weighted Hessian bound or stock smooth fit.
Our proof does not assume a derivative trace of the value at zero
variance. The local elliptic regularity results we use apply at positive
variance. The axis equation is obtained separately by a weighted
kernel representation.

We now state the main results. Theorem~\ref{thm:sharp} gives the exact
approach rate and its leading coefficient. Theorem~\ref{thm:main}
establishes the endpoint and the boundary regularity.
Corollary~\ref{cor:boundarycurve} expresses that regularity directly in
the stock--variance plane.

For $\nu>1$ set
\[
 J_\nu(c)=\int_0^\pi e^{c\theta}\sin^{\nu-2}\theta d\theta.
\]
For the critical case set $E=e^{-\pi c}$, $C_c=1+e^{\pi c}$, and
\begin{equation}\label{eq:criticalconstants}
 A_c=\lim_{\epsilon\searrow0}\left\{
 \int_\epsilon^{\pi-\epsilon}\frac{e^{-c\theta}}{\sin\theta} d\theta
                       +(1+E)\log\epsilon\right\},\qquad
 k_c=\frac{A_c}{1+E}.
\end{equation}
The endpoint expansions of the integrand show that this finite part
exists. In particular, $C_0=2$ and $k_0=\log2$.

The first main result gives the exact equivalents of the logarithmic
gap $d(v)=\log(B/b(v))$ in the three regimes.

\begin{theorem}\label{thm:sharp}
Under \eqref{eq:parameters}, the logarithmic gap $d$ has the following
exact asymptotic equivalents as $v\searrow0$.

If $\nu>1$, then
\begin{equation}\label{eq:sharpsuper}
d(v)\sim\Lambda^{-1}
 \left\{\frac{\nu-1}{J_\nu(c)}\right\}^{\frac1{\nu-1}}
 v\Bigl(\log\frac1v\Bigr)^{\frac1{\nu-1}}.
\end{equation}
If $\nu=1$, then
\begin{equation}\label{eq:sharpcritical}
d(v)\sim\Lambda^{-1}e^{1-k_c-\frac1{C_c}}
                  v\exp\sqrt{\frac2{C_c}\log\frac1v}.
\end{equation}
If $0<\nu<1$, let $\beta\in(\nu,1)$ be the unique solution of
$\sin\pi(\beta-\nu)=e^{-\pi c}\sin\pi\beta$. There is
$A\in(0,\infty)$ such that
\begin{equation}\label{eq:sharpsub}
d(v)\sim\Lambda^{-1}A^{-\frac1\beta}v^{\frac\nu\beta}.
\end{equation}
The coefficient $A$ is determined by the convergent integral
\eqref{eq:amplitudeformula}, with the density \eqref{eq:actualM}.
This representation involves the unknown boundary profile and is not a
formula in the model parameters alone. In all three cases,
$B-b(v)\sim Bd(v)$.
\end{theorem}

For $r=q$ and $\rho=0$, the first two formulas become
\[
d(v)\sim\frac v{\pi\sigma}\log\frac1v\quad(\nu=2),\qquad
d(v)\sim\frac{\sqrt e}{2\sigma}\,v\exp\sqrt{\log\frac1v}
                                                    \quad(\nu=1).
\]
For $0<\nu<1$ in the same symmetric case, the exponent is
$2\nu/(1+\nu)$. The coefficients for $\nu\ge1$ depend only on
the local model parameters. In the power regime, the coefficient is
characterized by a convergent integral involving the unknown boundary profile.
We do not give a formula in the model parameters alone.

The geometric properties used to prove the exact equivalents also give
the following endpoint regularity result. We prove it independently
of the coefficient selection.

\begin{theorem}\label{thm:main}
Assume \eqref{eq:parameters}, and let $m=\min\{\nu,1\}$. Then the
following statements hold.
\begin{enumerate}[label=\textup{(\roman*)}]
\item For every finite $v\ge0$ one has $0<b(v)<K$, and $b$ is
continuous and strictly decreasing on $[0,\infty)$.
\item The stock gap $B-b(v)$ is of larger order than the variance:
\begin{equation}\label{eq:mainseparation}
 \frac{B-b(v)}v\longrightarrow\infty\qquad(v\searrow0).
\end{equation}
\item For each $0<\alpha<m$, there are $C_\alpha,v_\alpha>0$ such that
\begin{equation}\label{eq:mainpower}
 B-b(v)\le C_\alpha v^\alpha,\qquad 0<v<v_\alpha.
\end{equation}
If $\nu\ge1$, then
\begin{equation}\label{eq:mainlogexponent}
 \lim_{v\searrow0}\frac{\log d(v)}{\log v}=1.
\end{equation}
\item The function $b$ is real analytic on $(0,\infty)$, with $b'<0$
and
\begin{equation}\label{eq:mainactualslope}
 b'(v)\longrightarrow-\infty\qquad(v\searrow0).
\end{equation}
\item The function $d$ has a strictly increasing local inverse
$\gamma$, whose extension by $\gamma(0)=0$ belongs to
$C^1([0,\delta_0))$ for some $\delta_0>0$, with $\gamma'(0)=0$.
If $\nu\ge1$, this extension is not $C^{1,\alpha}$ for any
$0<\alpha\le1$.
\end{enumerate}
\end{theorem}

The endpoint difference quotient in \eqref{eq:mainseparation} does
not by itself imply continuity of $\gamma'$. For that conclusion,
we rescale about moving contact points by the local variance. The
operators remain uniformly elliptic near rescaled variance one, and
the horizontal derivative tends to zero. Uniform free-boundary
estimates then show that the normal becomes vertical.
Theorem~\ref{thm:endpointregularity} gives a modulus in terms of
$\gamma(\delta)/\delta$. We do not differentiate the asymptotic
equivalents in Theorem~\ref{thm:sharp}.

We use the term regular $C^1$ curve for the image of an injective
$C^1$ parametrization whose tangent vector never vanishes. The
following consequence makes explicit the geometric content of
Theorem~\ref{thm:main}.

\begin{corollary}\label{cor:boundarycurve}
Under \eqref{eq:parameters}, the exercise boundary near $(B,0)$
is a regular $C^1$ curve tangent to the axis $v=0$. It is real
analytic away from the endpoint. If $\nu\ge1$, it admits no regular
$C^{1,\alpha}$ parametrization up to the endpoint for any
$0<\alpha\le1$.
\end{corollary}

\begin{proof}
Let $\delta_0$ be as in Theorem~\ref{thm:main}. The local boundary
has the parametrization
\begin{equation}\label{eq:boundarycurve}
 p(\delta)=\bigl(Be^{-\delta},\gamma(\delta)\bigr),
 \qquad 0\le\delta<\delta_0.
\end{equation}
It is injective, belongs to $C^1$, and satisfies
\[
 p'(\delta)=\bigl(-Be^{-\delta},\gamma'(\delta)\bigr),
 \qquad p'(0)=(-B,0)\ne0.
\]
Thus its tangent at the endpoint is horizontal. Analyticity at
positive variance follows from Theorem~\ref{thm:main}.

Suppose that $\nu\ge1$ and that another regular parametrization
$\widetilde p(t)$ were $C^{1,\alpha}$ at the endpoint. Its first
component would have a nonzero derivative there, since the tangent
is horizontal. The change of parameter
$\delta(t)=\log(B/\widetilde p_1(t))$ would therefore have a
$C^{1,\alpha}$ local inverse. Composing this inverse with
$\widetilde p_2$ would make $\gamma$ belong to $C^{1,\alpha}$,
contrary to Theorem~\ref{thm:main}.
\end{proof}

The divergence of $b'$ concerns a different parametrization of the
same curve. Differentiating $\gamma(d(v))=v$ for $v>0$ gives
\begin{equation}\label{eq:inverseboundaryderivative}
 b'(v)=-\frac{b(v)}{\gamma'(d(v))}\longrightarrow-\infty.
\end{equation}
With stock on the horizontal axis, $b'$ is the derivative $ ds/ dv$,
and the tangent slope $ dv/ ds$ tends to zero. For example, the smooth
arc $v=(B-s)^2$, $s\le B$, is represented in the opposite direction
by $s=B-\sqrt v$, whose derivative diverges at zero. For $0<\nu<1$,
we prove the $C^1$ conclusion above but do not determine the optimal
H\"older regularity of the tangent.

We now outline the proofs. The argument has three parts. First, the
linear excess-value bound in
Proposition~\ref{prop:contactbound} and the product barrier in
Proposition~\ref{prop:upper} give continuity of $b$ at zero. A
nondegeneracy estimate and a small-time entry argument then prove
$d(v)/v\to\infty$ in Theorem~\ref{thm:separation}. The entry argument
is used under a contradiction hypothesis that the ratio is bounded, so it does
not require stock smooth fit at zero variance as an input.
Theorem~\ref{thm:endpointregularity} obtains continuity of the tangent
by rescaling at moving contact points. This completes the geometric
part of the argument before the exact rates are determined.

Second, vertical integration of the obstacle equation controls the
error after averaging in log-stock. Theorem~\ref{thm:anisocompact}
uses convexity to obtain compactness of the rescaled inverse boundaries
and concavity of their normalized limits. Lemma~\ref{lem:axislinear}
gives an $O(v)$ approach estimate for the value on the whole axis.
After a shear, Proposition~\ref{prop:axisequation} represents the
axis trace by an explicit adjoint kernel. The lower-order terms
have a tangential $C^{1,\lambda}$ remainder for every $\lambda<1$.

Third, this equation selects the rate. For $\nu>1$, the contribution
near the observation point balances the derivative tail and gives a
logarithmic power in Proposition~\ref{prop:superrate}. At $\nu=1$,
the local integral has a logarithmic divergence, and the finite-part
calculation in Proposition~\ref{prop:criticalrate} determines the
coefficient. For $0<\nu<1$, the thin-layer limit is a nonsymmetric
Riesz potential. In Proposition~\ref{prop:subcriticalratios}, convexity
of the entire axis trace supplies the growth and integrability
conditions for a stable half-line classification. Lemma~\ref{lem:mellinamplitude}
then proves convergence to a positive finite amplitude by a Mellin
transform, and Proposition~\ref{prop:subcriticalrate} applies this
argument to the actual boundary. The endpoint regularity theorem is
not an input to these coefficient calculations.

Section~\ref{sec:prelim} gives the value estimates.
Section~\ref{sec:geometry} establishes the boundary geometry, scale
separation and endpoint regularity. Section~\ref{sec:axis} contains
the vertical identities, the compactness of boundary profiles and the
axis equation, and Section~\ref{sec:rates} proves the exact
equivalents. Appendix~\ref{sec:explicitcusp} discusses joint
differentiability at the contact point and gives an explicit
logarithmic contact curve.

\section{The value function and bounds at zero variance}\label{sec:prelim}

We use the parameters in \eqref{eq:notation} and write
$\HH=\R\times(0,\infty)$. On all $s>0$, the excess value is
\begin{equation}\label{eq:excess}
 R(s,y)=V(s,y)-K+s,\qquad w(x,y)=R(e^x,y),\qquad x_0=\log B.
\end{equation}
The isotropic rescaling and its limiting operator are
\begin{equation}\label{eq:blowup}
 W_\varepsilon(X,Y)=\frac{w(x_0+\varepsilon X,\varepsilon Y)}\varepsilon,
 \qquad f_0=rK-qB>0.
\end{equation}
\begin{equation}\label{eq:Lzero}
 \mathcal L_0=\frac Y2\partial_{XX}+\rho\sigma Y\partial_{XY}
       +DY\partial_{YY}+(r-q)\partial_X+a\partial_Y.
\end{equation}
The constant $f_0$ is positive by \eqref{eq:parameters} and $B<K$.
Constants in estimates may change from line to line and depend on the
fixed model parameters and stated localization domains. The function
$\gamma$ denotes the local inverse of $d$, whose existence follows
from the boundary properties proved in Section~\ref{sec:geometry}.

For positive functions, $F\sim G$ means $F/G\to1$ at the stated
limit, and $F\asymp G$ means that this ratio is bounded above and
below by positive constants. A positive measurable function $F$ is
regularly varying at zero with index $p$ if
\[
 \frac{F(tu)}{F(t)}\longrightarrow u^p\qquad(t\searrow0)
 \quad\hbox{for every }u>0.
\]
We use the uniform convergence theorem and the Potter bounds of
\cite[Chapter~1]{BGT} in this form at zero. In particular, for each
$\epsilon>0$ there are $C_\epsilon,t_\epsilon>0$ such that
\begin{equation}\label{eq:potterbounds}
 C_\epsilon^{-1}\min\{u^{p-\epsilon},u^{p+\epsilon}\}
 \le\frac{F(tu)}{F(t)}
 \le C_\epsilon\max\{u^{p-\epsilon},u^{p+\epsilon}\}
\end{equation}
whenever $0<t,tu<t_\epsilon$. In the boundary analysis below,
regular variation is proved before these bounds are applied.

We first record the basic properties of the value function.

\begin{lemma}\label{lem:basic}
The function $V$ is continuous on $(0,\infty)\times[0,\infty)$,
nonincreasing and convex in $s$, and nondecreasing in $v$. It satisfies
$0\le V\le K$ and $V\ge(K-s)^+$. For every $v\ge0$, the function $R$
in \eqref{eq:excess} is nonnegative, convex, nondecreasing, and
$1$-Lipschitz in $s$. It extends to $R(0,v)=0$, and
\begin{equation}\label{eq:convexexcess}
 0\le R(s,v)\le s,\qquad sR_s(s,v)\ge R(s,v)
\end{equation}
at points where the derivative exists.
\end{lemma}

\begin{proof}
Let $P_T$ be the value with remaining maturity $T$. Truncating any stopping
time at $T$ gives
\begin{equation}\label{eq:finiteapprox}
 0\le V-P_T\le Ke^{-rT}
\end{equation}
uniformly over all initial states. Indeed, on $\{\tau>T\}$ the original
discounted payoff is at most $Ke^{-rT}$, whereas the payoff at $T$ is
nonnegative. Continuity and monotonicity in the variance therefore follow
from the finite-maturity results in \cite[Proposition 4.1]{LT}.

For a fixed variance path, write $S_t=sZ_t$ using a common Brownian
filtration for all $s$. For each stopping time the expected payoff is
convex and nonincreasing in $s$. Taking the supremum preserves these
properties. The immediate exercise strategy and bounded payoff give
$(K-s)^+\le V(s,v)\le K$, so that $0\le R(s,v)\le s$ and $R(0,v)=0$.

Convexity implies that $R(s,v)/s$ is nondecreasing and that
$sR_s\ge R$ wherever $R_s$ exists. It also implies monotonicity of $R$:
for $0<s_1<s_2$,
$R(s_1,v)\le(s_1/s_2)R(s_2,v)\le R(s_2,v)$.
Finally, the monotonicity of $V$ gives
$R(s_2,v)-R(s_1,v)\le s_2-s_1$.
\end{proof}

In particular, the zero set of $R(\cdot,v)$ is an initial interval, and
$b$ is nonincreasing in $v$. We next state the local equation, including
the inequality across the contact set.

The discounted generator in logarithmic coordinates is
\begin{equation}\label{eq:operator}
 \Lop=\frac y2\partial_{xx}+\rho\sigma y\partial_{xy}
 +Dy\partial_{yy}+\Bigl(r-q-\frac y2\Bigr)\partial_x+(a-\kappa y)\partial_y-r.
\end{equation}
Put $f(x)=rK-qe^x$ and
$\CC=\{(x,y):y>0,\ V(e^x,y)>(K-e^x)^+\}$.

\begin{lemma}\label{lem:pde}
The function $w$ is smooth in $\CC$ and satisfies $\Lop w=f$ there.
On $\{x<\log K,y>0\}$ it solves, in the viscosity sense,
\begin{equation}\label{eq:vi}
 \min\{f-\Lop w,w\}=0.
\end{equation}
In particular, if a smooth function $\phi$ touches $w$ from below at a
point with positive variance, then $\Lop\phi\le f$ at that point.
\end{lemma}

\begin{proof}
The discounted reward is continuous, bounded by $K$, and converges to zero
at infinite time. The Snell envelope is therefore of class D. The Markov
optimal stopping principle gives
\begin{equation}\label{eq:excessive}
 V(s,v)\ge\E_{s,v}[e^{-r\eta}V(S_\eta,Y_\eta)]
\end{equation}
for bounded stopping times $\eta$, and equality when stopped before exit
from a relatively compact subset of the continuation region. These are
the usual local properties of the Snell envelope~\cite{PS}.
Local uniform ellipticity at $y>0$ identifies $V$ with the local
Dirichlet solution in the continuation region and gives its smoothness.
Since $\Lop(K-e^x)=-f(x)$, the equation for $w$ follows.

If $\phi$ touches $w$ from below, then $\phi+K-e^x$ touches $V(e^x,y)$
from below. Apply \eqref{eq:excessive} at the minimum of a small time and
the exit time from a small positive-variance ball. Ito's formula and
division by the small time give $\Lop\phi-f\le0$. On the contact set
with $x<\log K$, $w=0$. In the continuation set the equation holds.
Together these give \eqref{eq:vi}.
\end{proof}

We use the following local consequence of classical obstacle regularity.
On a relatively compact positive-variance domain on which the obstacle
is smooth, $w\in W^{2,p}_{\loc}$ for every finite $p$. The local estimate
depends on ellipticity and coefficient bounds, the source, and a local
bound for $w$. We use \cite[Theorems A.3(2) and A.5]{Feehan}, with
$A=-\Lop$, source $-f$, and obstacle zero. Local Dirichlet comparison
identifies that solution with the viscosity solution above. We will apply
these estimates only on domains where their ellipticity constants are
uniformly positive.

\begin{lemma}\label{lem:positiveboundary}
For every finite $v\ge0$, one has $0<b(v)<K$.
\end{lemma}

\begin{proof}
For any positive deterministic time, the integrated variance is positive
almost surely, including when $v=0$. Conditionally on the Brownian motion
driving $Y$, the independent Brownian component of $\log S$ has strictly
positive variance because $|\rho|<1$. Thus a European put has positive
value for every finite initial stock price $s>0$. In particular,
$V(K,v)>0$, and continuity of $V$ gives $b(v)<K$.

Suppose $b(v_*)=0$. Variance monotonicity gives $b(y)=0$ for all $y>v_*$.
On any positive-variance interval $J\Subset(v_*,\infty)$, the whole strip
$\R\times J$ is therefore a continuation region. On this strip,
$w_n(x,y)=w(x-n,y)$ satisfies
\[
 \Lop w_n=rK-qe^{x-n},\qquad 0\le w_n(x,y)\le e^{x-n}.
\]
The coefficients of $\Lop$ are independent of $x$. Testing this equation
against a fixed compactly supported smooth function and moving all
derivatives onto the test function gives, as $n\to\infty$, the
distributional identity $0=rK$. This contradicts $rK>0$.
\end{proof}

The excess value grows at most linearly in the variance below the
zero-variance threshold.

\begin{proposition}\label{prop:contactbound}
There are $C,y_0>0$ such that
\begin{equation}\label{eq:contactbound}
 0\le R(s,y)\le Cy,\qquad 0<s\le B,\quad 0<y<y_0.
\end{equation}
Consequently, for every $x\in\R$ and $0<y<y_0$,
\begin{equation}\label{eq:globalcontactbound}
 0\le w(x,y)\le Cy+(e^x-B)^+.
\end{equation}
\end{proposition}

\begin{proof}
Start the process at $(B,0)$. Since $V(B,0)=K-B$, the dynamic programming
inequality and the nonnegative supermartingale given by the discounted stock give
\begin{align}
 \E_{B,0}[e^{-rh}R(S_h,Y_h)]
 &\le K(1-e^{-rh})-B(1-e^{-qh})\notag\\
 &\le rKh.\label{eq:contactdpp}
\end{align}
The CIR moment equations at $Y_0=0$ yield
\[
 m_h:=\E Y_h=\theta(1-e^{-\kappa h}),\qquad
 \E Y_h^2=\theta\left(\theta+\frac{\sigma^2}{2\kappa}\right)
                  (1-e^{-\kappa h})^2.
\]
Hence the Paley--Zygmund inequality gives
\begin{equation}\label{eq:pz}
 \PP\Bigl(Y_h\ge\frac{m_h}2\Bigr)\ge\frac{\nu}{4(\nu+1)}=:p_0>0.
\end{equation}
Choose $h=4y/a$. For $h\le1/\kappa$, $m_h\ge ah/2=2y$, and thus
$\PP(Y_h\ge y)\ge p_0$.

Let $Z_h=\log(S_h/B)$. Since $\E Y_t\le at$, Ito isometry gives
\begin{align*}
 \E|Z_h|
 &\le |r-q|h+\frac12\E\int_0^hY_t dt
          +\left(\E\int_0^hY_t dt\right)^{\frac12}\\
 &\le |r-q|h+\frac a4h^2+\sqrt{\frac a2}\,h\le C_1h
 \qquad(h\le1).
\end{align*}
Fix $M\ge2C_1/p_0$. Markov's inequality and \eqref{eq:pz} imply
\[
 \PP(S_h\ge Be^{-Mh},\ Y_h\ge y)\ge\frac{p_0}2.
\]
No independence of these two events is used. By monotonicity of $R$ in
both variables and \eqref{eq:contactdpp},
\[
 R(Be^{-Mh},y)\le\frac{2e^{rh}}{p_0}rKh.
\]
The stock Lipschitz bound now yields
\[
 R(B,y)\le\frac{2e^{rh}}{p_0}rKh+B(1-e^{-Mh})\le Cy.
\]
Monotonicity extends the estimate to $s\le B$. The same Lipschitz bound
gives \eqref{eq:globalcontactbound} for $s>B$.
\end{proof}

\begin{lemma}\label{lem:H}
In the continuation region, let $H=w_x-w$. Then
\begin{equation}\label{eq:Hequation}
 H\ge0,\qquad \Lop H=-rK.
\end{equation}
For $0<\delta\le1$, $x\le x_0-\delta/4$, and small $y>0$, whenever the
point belongs to the continuation region,
\begin{equation}\label{eq:Hupper}
 H(x,y)\le C\frac y\delta.
\end{equation}
\end{lemma}

\begin{proof}
The first assertion follows from \eqref{eq:convexexcess}. The coefficients
of $\Lop$ are independent of $x$, so
$\Lop H=(\partial_x-1)f=-rK$. The secant inequality for the convex
function $R$ gives
\[
 H\le w_x=e^xR_s(e^x,y)
 \le\frac{e^x}{B-e^x}R(B,y)\le C\frac y\delta.
\]
\end{proof}

\section{Geometry of the exercise boundary}\label{sec:geometry}

For $0<v\le1$, define
\begin{equation}\label{eq:omega}
 \omega_\nu(v)=
 \begin{cases}
 v^{\frac\nu{\nu+1}},&0<\nu<1,\\
 \Bigl(v\log\frac ev\Bigr)^{\frac12},&\nu=1,\\
 v^{\frac12},&\nu>1.
 \end{cases}
\end{equation}
The first estimates in this section establish continuity and strict
monotonicity of the boundary, together with the preliminary bound
\begin{equation}\label{eq:mainbounds}
 c_*v\le B-b(v)\le C\omega_\nu(v).
\end{equation}
The upper estimate first gives continuity at zero. We use that
continuity when deriving the lower estimate and the inverse boundary.
The rest of the section studies blowup limits, proves scale separation
and shows that the tangent is continuous at the endpoint.

We begin with the expected exit time of the variance process from an
interval. Let $\mathcal A=Dy\partial_{yy}+(a-\kappa y)\partial_y$ and fix
$L_0\le\min\{1,a/(2\kappa)\}$. For $0<v\le L/8$ and $L\le L_0$, let
$m=m_{v,L}$ solve
\begin{equation}\label{eq:exitode}
 \mathcal Am=-1\quad\hbox{on }(v,L),\qquad m(v)=m(L)=0.
\end{equation}
This is the expected exit time of $Y$ from $(v,L)$.

\begin{lemma}\label{lem:exit}
There are constants $c_*,C>0$, independent of $v$ and $L$, such that
\begin{equation}\label{eq:exitupper}
 0\le m\le CL,\qquad y|m'(y)|\le CL\quad(v\le y\le L),
\end{equation}
and
\begin{equation}\label{eq:exitlower}
 m(2v)\ge c_*L\Phi_\nu\Bigl(\frac vL\Bigr),\qquad
 \Phi_\nu(t)=
 \begin{cases}
 t^{1-\nu},&0<\nu<1,\\
 \Bigl(\log\frac1t\Bigr)^{-1},&\nu=1,\\
 1,&\nu>1.
 \end{cases}
\end{equation}
\end{lemma}

\begin{proof}
The maximum principle gives $m\ge0$. The function $2(L-y)/a$ satisfies
$\mathcal A[2(L-y)/a]\le-1$ and dominates the boundary data, so
$m\le2L/a$.

Put $\lambda=\kappa/D$ and
\[
 p(y)=y^{-\nu}e^{\lambda y},\qquad
 J(y)=\frac1D\int_0^y t^{\nu-1}e^{-\lambda t} dt.
\]
An integrating factor in \eqref{eq:exitode} gives
\begin{equation}\label{eq:exitderivative}
 m'(y)=p(y)(C_{v,L}-J(y)),\qquad
 C_{v,L}=\frac{\int_v^L p(t)J(t) dt}{\int_v^L p(t) dt}.
\end{equation}
Here $0\le J(y)\le y^\nu/a$ and $0\le C_{v,L}\le J(L)$.
For $0<\nu\le1$, both $y^{1-\nu}J(y)$ and
$y^{1-\nu}C_{v,L}$ are bounded by $CL$. For $\nu>1$,
\[
 \int_v^L pJ dt\le CL,\qquad
 \int_v^L p dt\ge\int_v^{2v}t^{-\nu} dt=c_\nu v^{1-\nu}.
\]
Thus $C_{v,L}\le CLv^{\nu-1}$, and
$y^{1-\nu}C_{v,L}\le CL(v/y)^{\nu-1}\le CL$. Since the exponential
factor is uniformly bounded for $y\le L_0$, this proves
\eqref{eq:exitupper}.

For the lower bound, first note that $m(L/2)\ge c_*L$. Indeed, on
$(L/4,3L/4)$ the function
$\zeta(y)=(y-L/4)(3L/4-y)/L$ vanishes at the endpoints,
satisfies $|\mathcal A\zeta|\le C$, and has $\zeta(L/2)=L/16$.
Comparison with the mean exit time from that smaller interval proves
the claim.

The scale density is $p$, and hence
\[
 \PP_{2v}(T_{\frac L2}<T_v)
 =\frac{\int_v^{2v}p(t) dt}{\int_v^{\frac L2}p(t) dt}.
\]
The strong Markov property gives
$m(2v)\ge m(L/2)\PP_{2v}(T_{L/2}<T_v)$. On $(0,L_0)$ the factor
$e^{\lambda t}$ is bounded above and below. Direct integration of
$t^{-\nu}$ therefore gives \eqref{eq:exitlower} in the three cases.
\end{proof}

The exit time gives a barrier in both variables. Fix $x_c\in\R$ and
let
\[
 Q=(x_c-L,x_c+L)\times(v,L),\qquad
 \phi(x)=1-\frac{(x-x_c)^2}{L^2},\qquad \psi(x,y)=\phi(x)m(y).
\]
Using \eqref{eq:exitode}, we obtain
\begin{equation}\label{eq:product}
 \Lop\psi=-\phi+\frac y2\phi''m+\Bigl(r-q-\frac y2\Bigr)\phi'm
            +\rho\sigma y\phi'm'-r\phi m.
\end{equation}
Since $|\phi'|\le2/L$, $|\phi''|=2/L^2$, and $y\le L$,
Lemma~\ref{lem:exit} bounds every term on the right independently of
$x_c,v,L$. Thus, for a fixed $C_*>0$,
\begin{equation}\label{eq:barrier}
 \beta:=\frac\psi{C_*}\ge0,\qquad
 \beta=0\text{ on }\partial Q,\qquad \Lop\beta\ge-1.
\end{equation}
The weighted derivative $y m'$ in \eqref{eq:exitupper} is precisely what
controls the mixed term in \eqref{eq:product}.

For example, if $\tau_Q$ is the first exit from $Q$, stopped Ito's formula
gives
\begin{equation}\label{eq:2dexit}
 \E_{x_c,2v}\int_0^{\tau_Q}e^{-rt} dt
 \ge\beta(x_c,2v)=\frac{m(2v)}{C_*}
 \ge c_*L\Phi_\nu\Bigl(\frac vL\Bigr).
\end{equation}
The exit time has finite mean because it is bounded above by the exit
time of $Y$ from $(v,L)$.

\begin{proposition}\label{prop:upper}
For sufficiently small $v>0$, $B-b(v)\le C\omega_\nu(v)$. In particular,
$b(v)\to B$ as $v\searrow0$.
\end{proposition}

\begin{proof}
By Lemma~\ref{lem:positiveboundary}, $d(v)=\log(B/b(v))$ is finite and
nonnegative. Choose $0<\delta_*\le\min\{1,4L_0\}$ and set
\[
 \delta=\min\{d(v),\delta_*\},\qquad L=\frac\delta4,
 \qquad x_c=x_0-\frac\delta2.
\]
If $\delta=0$ there is nothing to prove. If $L<8v$, then $\delta<32v$.
For the rest of the proof, assume $8v\le L$. The closed rectangle
\[
 \overline Q=\Bigl[x_0-\frac{3\delta}4,\,x_0-\frac\delta4\Bigr]\times[v,L]
\]
lies strictly inside the continuation region: for $y\ge v$,
$b(y)\le b(v)$, while $x_0-3\delta/4>x_0-d(v)$.

By \eqref{eq:Hequation} and \eqref{eq:barrier},
$\Lop(H-rK\beta)\le0$, and $H-rK\beta\ge0$ on $\partial Q$.
The minimum principle yields $H\ge rK\beta$. At $(x_c,2v)$,
\eqref{eq:Hupper} and Lemma~\ref{lem:exit} therefore give
\begin{equation}\label{eq:keygap}
 C\frac v\delta\ge c_*L\Phi_\nu\Bigl(\frac vL\Bigr).
\end{equation}
Since $L=\delta/4$, this implies, respectively,
\[
 \delta^{\nu+1}\le Cv^\nu\quad(\nu<1),\qquad
 \delta^2\le Cv\log\frac ev\quad(\nu=1),\qquad
 \delta^2\le Cv\quad(\nu>1).
\]
In the middle case we used $L\le1$. The case $\delta<32v$ satisfies the
same bounds after enlarging the constants. These bounds tend to zero, so
eventually $\delta<\delta_*$ and consequently $\delta=d(v)$.
Finally, $B-b(v)=B(1-e^{-d(v)})\le Bd(v)$. This also proves continuity
at zero, without using it in the argument.
\end{proof}

We next give a local observation used at positive variance. It avoids
assuming smoothness of a general free boundary.

\begin{lemma}\label{lem:flat}
Let $P=A_{ij}\partial_{ij}+B_i\partial_i-c_P$ be a smooth uniformly
elliptic operator in a rectangle with coordinates $(\xi,\eta)$, and let
$F$ be smooth. Suppose $u$ is continuous and nonnegative,
\[
 u=0\quad\hbox{for }\xi\le0,\qquad Pu=F\quad\hbox{for }\xi>0,
\]
and $Pu\le F$ across the rectangle, in the sense of smooth test functions
touching from below. Then, locally on the interface,
\begin{equation}\label{eq:flatjet}
 u_\xi(0+,\eta)=0,\qquad
 u_{\xi\xi}(0+,\eta)=\frac{F(0,\eta)}{A_{\xi\xi}(0,\eta)}.
\end{equation}
All first derivatives and the tangential and mixed second derivatives
vanish there.
\end{lemma}

\begin{proof}
On the side $\xi>0$, local boundary regularity for a smooth uniformly
elliptic Dirichlet problem gives smoothness up to the flat boundary,
whose boundary data are zero, by \cite[Chapter 6]{GT}. This is a local
statement away from the other sides of the rectangle.
Let $p(\eta)=u_\xi(0+,\eta)$. Nonnegativity gives $p\ge0$.
If $p(\eta_0)>0$, choose $c_0>0$ and an interval around $\eta_0$ on which
$p\ge2c_0$. For any fixed $M>0$, shrinking the rectangle if necessary
makes
\[
 \phi(\xi,\eta)=c_0\xi+M\xi^2
\]
a lower test function at $(0,\eta_0)$. For small negative $\xi$ it is
nonpositive. For positive $\xi$, the boundary Taylor expansion gives
$u\ge(3c_0/2)\xi\ge\phi$ after a further shrinking. At the contact point,
\[
 P\phi=2M A_{\xi\xi}+c_0 B_\xi.
\]
Taking $M$ large contradicts $P\phi\le F$. Thus $p=0$ identically along
the interface. Tangential differentiation of $u=0$ and $u_\xi=0$ gives
the asserted tangential and mixed derivatives. Evaluating $Pu=F$ from
the positive side now gives \eqref{eq:flatjet}.
\end{proof}

\begin{proposition}\label{prop:strict}
The function $b$ is strictly decreasing on $[0,\infty)$ and continuous
on $(0,\infty)$.
\end{proposition}

\begin{proof}
Suppose first that $b(v_1)=b(v_2)=s_*$ for $0\le v_1<v_2$.
Monotonicity makes $b$ constant on this interval. Choose a smaller
positive-variance interval and write $x_*=\log s_*$. Locally, $w=0$
for $x\le x_*$, while $\Lop w=f$ for $x>x_*$. Lemmas~\ref{lem:pde}
and \ref{lem:flat} give
\begin{equation}\label{eq:verticaljet}
 w(x_*+h,y)=\frac{f(x_*)}{y}h^2+o(h^2)
 \qquad(h\searrow0).
\end{equation}
Here $f(x_*)>0$ by \eqref{eq:parameters} and $s_*<K$. For two fixed
heights $y_1<y_2$ in the interval, \eqref{eq:verticaljet} gives
$w(x_*+h,y_1)>w(x_*+h,y_2)$ for small $h>0$. This contradicts variance
monotonicity and proves strict decrease.

Left continuity at $v_*>0$ follows from closedness of the contact set.
More explicitly, if $v_n\nearrow v_*$, then every
$s<\lim_n b(v_n)$ is a contact point for each $v_n$, and continuity of
$V$ makes it a contact point at $v_*$. Monotonicity gives the reverse
inequality for the boundary limit.

If a right jump occurs, choose a log-stock interval strictly between
$\log b(v_*+)$ and $\log b(v_*)$. On this interval the region below
$y=v_*$ is contact and the region above it is continuation. Applying
Lemma~\ref{lem:flat} in the variance direction gives
\[
 w(x,v_*+h)=\frac{f(x)}{2Dv_*}h^2+o(h^2).
\]
The one-sided smoothness just used also permits tangential
differentiation of this expansion, locally uniformly in $x$. Hence
\[
 H(x,v_*+h)=\frac{f'(x)-f(x)}{2Dv_*}h^2+o(h^2)
          =-\frac{rK}{2Dv_*}h^2+o(h^2)<0
\]
for small $h>0$, contradicting \eqref{eq:Hequation}. There is no right
jump.
\end{proof}

The boundary gap is at least of the order of the variance.

\begin{proposition}\label{prop:lower}
One has
\begin{equation}\label{eq:liminflin}
 \liminf_{v\searrow0}\frac{d(v)}v>0.
\end{equation}
In particular, $B-b(v)\ge c_*v$ for small $v>0$.
\end{proposition}

\begin{proof}
Suppose instead that $\varepsilon_n\searrow0$ and
$d(\varepsilon_n)/\varepsilon_n\to0$. Define $W_n$ by
\eqref{eq:blowup} with $\varepsilon=\varepsilon_n$, and restrict for now
to $0<Y<1$. Since $R(b(y),y)=0$, stock Lipschitz continuity and boundary
monotonicity give
\begin{equation}\label{eq:smalltrace}
 0\le W_n(0,Y)
 \le\frac{B-b(\varepsilon_nY)}{\varepsilon_n}
 \le\frac{B d(\varepsilon_n)}{\varepsilon_n}\longrightarrow0
\end{equation}
uniformly for $0<Y\le1$. By stock monotonicity the same upper bound
holds for $X\le0$. On each bounded $X$ interval the $W_n$ have a common
Lipschitz constant in $X$, because
\[
 |W_n(X_2,Y)-W_n(X_1,Y)|
 \le\frac B{\varepsilon_n}|e^{\varepsilon_nX_2}-e^{\varepsilon_nX_1}|.
\]

For $X>0,Y>0$, the original point is in the continuation region for all
large $n$ on any fixed compact set. Direct rescaling gives
\begin{align}
 \mathcal L_nW_n={}&rK-qB e^{\varepsilon_nX},\label{eq:scaledpde}\\
 \mathcal L_n={}&\frac Y2\partial_{XX}+\rho\sigma Y\partial_{XY}
 +DY\partial_{YY}+\Bigl(r-q-\frac{\varepsilon_nY}2\Bigr)\partial_X\notag\\
 &+(a-\kappa\varepsilon_nY)\partial_Y-\varepsilon_nr.\notag
\end{align}
The bounds above and interior elliptic estimates give a smoothly
convergent subsequence on $\{X>0,0<Y<1\}$. Extend its limit $W$ by zero
for $X\le0$. The common stock Lipschitz bound and \eqref{eq:smalltrace}
give local uniform convergence on the whole strip.

Viscosity stability in Lemma~\ref{lem:pde} yields
$\mathcal L_0W\le f_0$ throughout the strip, and
$\mathcal L_0W=f_0$ for $X>0$. Also $W$ is nondecreasing in $Y$.
Lemma~\ref{lem:flat}, now applied at $X=0$ and away from $Y=0$, gives
\[
 W(X,Y)=\frac{f_0}{Y}X^2+o(X^2)\qquad(X\searrow0).
\]
For fixed $0<Y_1<Y_2<1$, the leading coefficients contradict
$W(X,Y_1)\le W(X,Y_2)$. This proves \eqref{eq:liminflin}.
Finally, Proposition~\ref{prop:upper} gives $d(v)\to0$, and
$B-b(v)=B(1-e^{-d(v)})\ge Bd(v)/2$ for small $v$.
\end{proof}

\begin{proof}[Derivation of the preliminary bound \eqref{eq:mainbounds}]
Combine Lemma~\ref{lem:positiveboundary} and
Propositions~\ref{prop:upper}, \ref{prop:strict}, and \ref{prop:lower}.
Continuity and strict monotonicity give the continuous inverse of $d$ on
a small interval. The inequality $d(v)\ge c_*v$ implies
$\gamma(\delta)\le C\delta$ there.
\end{proof}

At positive variance the boundary is real analytic and strictly
decreasing.

\begin{proposition}\label{prop:positiveregularity}
The boundary $b$ is real analytic on $(0,\infty)$, and $b'(v)<0$
there. Its local inverse $\gamma$ is real analytic at every positive
argument in its domain.
\end{proposition}

\begin{proof}
Fix a free-boundary point $p=(\log b(v_*),v_*)$ with $v_*>0$.
Work in a ball about $p$ contained in $\{x<\log K,y>0\}$.
Dividing the equation by $y$ gives
\begin{equation}\label{eq:constantprincipal}
\begin{gathered}
 A_0:D^2w=F\chi_{\{w>0\}},\qquad
 A_0=\begin{pmatrix}\frac12&\frac{\rho\sigma}2\\[2pt]\frac{\rho\sigma}2&D\end{pmatrix},\\
 F=\frac{f-\Bigl(r-q-\frac y2\Bigr)w_x-(a-\kappa y)w_y+rw}{y}.
\end{gathered}
\end{equation}
The matrix $A_0$ is positive definite. Local $W^{2,p}$ estimates
give $w\in C^{1,\alpha}$, and $F\in C^\alpha$ for any fixed
$\alpha<1$. Since $w=\nabla w=0$ at $p$, $F(p)=f(p_x)/v_*>0$.
After shrinking the ball, $F$ has a positive lower bound.

For completeness, the required local $C^{1,1}$ estimate follows from
\eqref{eq:constantprincipal}. At a contact point $z$, the
inhomogeneous Harnack inequality, applied on a ball of radius $2h$,
gives $\sup_{B_h(z)}w\le Ch^2$, since $w\ge0$, $w(z)=0$,
and the right-hand side is bounded. At a continuation point a
distance $h$ from contact, the interior Schauder estimate on a ball
of radius $h/2$ then bounds $D^2w$, using the $C^\alpha$ bound for
$F$. The Hessian vanishes almost everywhere in the contact set.
Thus $w$ is locally $C^{1,1}$ across the boundary.

Monotonicity in both variables gives the following boundary
chart. In the rotated coordinates
\[
 \xi=\frac{(x-p_x)-(y-p_y)}{\sqrt2},\qquad
 \zeta=\frac{(x-p_x)+(y-p_y)}{\sqrt2},
\]
the contact set lies below a $1$-Lipschitz graph through the origin,
and the continuation set lies above it. Indeed, contact is preserved
by decreasing both original coordinates, and continuation by
increasing them. These two cones give the asserted Lipschitz bound.
Fixing the H\"older exponent at $1/2$, the positive $C^{1/2}$ source,
the $C^{1,1}$ bound, and the zero Cauchy data verify hypotheses
H1--H2 in \cite[Section 1.2]{Caffarelli}, after a fixed linear change
of variables normalizes $A_0$. Caffarelli's boundary regularity theorem
\cite[Theorem 3]{Caffarelli} therefore gives a $C^1$ free boundary
and continuity of the second derivatives of $w$ from continuation.
The analytic elliptic equation $\Lop w=f$ has zero Cauchy data on
this boundary and a nonzero second normal derivative.
Theorem~1$'$ of Kinderlehrer and Nirenberg~\cite{KN} gives local
analyticity of the boundary and smoothness of $w$ from continuation.

It remains to check that the boundary is a graph in each original
coordinate with a nonzero slope. Let $n=(n_x,n_y)$ be its unit normal
pointing into continuation. Monotonicity gives $n_x,n_y\ge0$.
The zero Cauchy data and the equation imply
\begin{equation}\label{eq:normalhessian}
 D^2w(p)=\lambda_p n\otimes n,\qquad
 \lambda_p=\frac{f(p_x)}{v_* n^T A_0 n}>0,
\end{equation}
where derivatives are taken from continuation.
There $w_x\ge w>0$ by \eqref{eq:convexexcess}, and
\begin{equation}\label{eq:derivativeequations}
 \Lop w_x=-qe^x\le0,\qquad
 (\Lop+\rho\sigma\partial_x+D\partial_y-\kappa)w_y
       =-\tfrac12(w_{xx}-w_x)\le0.
\end{equation}
The second inequality uses stock convexity. The continuation set
with $x<\log K$ is connected. Strict decrease of $b$ shows that
$w_y$ is not identically zero there: between two distinct boundary
stock prices, increasing the variance changes zero excess value
to positive excess value. The strong minimum principle gives
$w_y>0$ throughout this set. Applying the Hopf lemma to $w_x$ and
$w_y$ at $p$ now gives
$\lambda_p n_x>0$ and $\lambda_p n_y>0$.
Consequently $b'(v_*)=-b(v_*)n_y/n_x<0$.
The analytic inverse function theorem proves the assertion for
$\gamma$.
\end{proof}

We now turn to blowup limits. The isotropic rescaling controls the
value on regions with positive rescaled variance. It will also be used
to exclude boundary gaps of order $v$ and to study the tangent at
moving contact points. The limits along subsequences have the
following properties.

\begin{proposition}\label{prop:compact}
Every sequence $\varepsilon_n\searrow0$ has a subsequence along which
$W_{\varepsilon_n}$ converges in $C^{1,\alpha}_{\loc}(\HH)$, for any fixed
$0<\alpha<1$, to a function $W$. The limit is nonnegative, nondecreasing
in both variables, and convex in $X$, and satisfies
\begin{equation}\label{eq:limitbounds}
 0\le W(X,Y)\le CY+BX^+,\qquad 0\le W_X\le B.
\end{equation}
It solves
\begin{equation}\label{eq:limitvi}
 \min\{f_0-\mathcal L_0W,W\}=0\quad\hbox{in }\HH,
\end{equation}
where $\mathcal L_0$ is defined in \eqref{eq:Lzero}.
Moreover, $W>0$ on $\{X>0,Y>0\}$.
\end{proposition}

\begin{proof}
Proposition~\ref{prop:contactbound} gives
\begin{equation}\label{eq:scaledbound}
 0\le W_\varepsilon(X,Y)
 \le CY+\frac{B(e^{\varepsilon X}-1)^+}{\varepsilon}
\end{equation}
on every fixed compact set for sufficiently small $\varepsilon$.
For such $\varepsilon$, the original stock is less than $K$ on that
compact set, so $W_\varepsilon$ solves
\[
 \min\{f_\varepsilon-\mathcal L_\varepsilon W_\varepsilon,
                 W_\varepsilon\}=0,
 \qquad f_\varepsilon=rK-qB e^{\varepsilon X},
\]
where $\mathcal L_\varepsilon$ is given in \eqref{eq:scaledpde}.
On any $E\Subset\HH$, the coefficients on a slightly larger domain are
uniformly smooth, the ellipticity constant is bounded below independently
of $\varepsilon$, and $f_\varepsilon$ is uniformly bounded.

The local obstacle estimate recalled in Section~\ref{sec:prelim} gives
$\|W_\varepsilon\|_{W^{2,p}(E)}\le C_{E,p}$ for every finite $p$.
The zero-order coefficient is $-\varepsilon r$. The local estimate uses
its nonpositive sign and its upper bound in magnitude, not a positive
lower bound on $\varepsilon r$.
Compact embedding, with $p$ chosen so that $\alpha<1-2/p$, followed by
a diagonal subsequence, gives convergence in $C^{1,\alpha}_{\loc}(\HH)$.

Monotonicity in $X$ and $Y$ passes to the limit. Stock convexity gives
\[
 W_{\varepsilon,XX}-\varepsilon W_{\varepsilon,X}\ge0
\]
in distributions, and hence $W_{XX}\ge0$ in distributions. Stock
Lipschitz continuity also gives
$0\le W_{\varepsilon,X}\le B e^{\varepsilon X}$, which yields the
derivative bound in \eqref{eq:limitbounds}.
The growth bound and the viscosity equation pass to the limit as well.

On $X>0,Y>0$, all original points are continuation points for small
$\varepsilon$, so $\mathcal L_0W=f_0$ there. If $W(X_*,Y_*)=0$ for
some $X_*,Y_*>0$, monotonicity and nonnegativity would force $W=0$ in
the rectangle $0<X<X_*$, $0<Y<Y_*$. This contradicts $f_0>0$.
\end{proof}

The convergence in Proposition~\ref{prop:compact} includes contact points
with $Y>0$. It does not include derivative convergence at $Y=0$ or
convergence of the free boundaries. The growth bound does give a
continuous zero trace on $\{X\le0,Y=0\}$.

The excess value at the threshold is of the exact order of the
variance.

\begin{lemma}\label{lem:contactnondegenerate}
There are $c_*,C,v_0>0$ such that
\begin{equation}\label{eq:contactnondegenerate}
 c_*v\le R(B,v)\le Cv,\qquad 0<v<v_0.
\end{equation}
\end{lemma}

\begin{proof}
The upper bound is Proposition~\ref{prop:contactbound}. If the lower
bound failed, there would be $\varepsilon_n\searrow0$ with
$W_{\varepsilon_n}(0,1)\to0$. Pass to a limit $W$ using
Proposition~\ref{prop:compact}. Nonnegativity and monotonicity give
$W=0$ on $\{X\le0,0<Y<1\}$. On $\{X>0,0<Y<1\}$, one has
$\mathcal L_0W=f_0$. Apply Lemma~\ref{lem:flat} locally along $X=0$.
For every fixed $Y\in(0,1)$,
\[
 W(X,Y)=\frac{f_0}{Y}X^2+o(X^2)\qquad(X\searrow0).
\]
At two heights $0<Y_1<Y_2<1$, this contradicts monotonicity in $Y$.
\end{proof}

The next estimate prepares the separation argument. We use the usual
completed Brownian filtration. Write $\mu=r-q$.
The estimate below concerns a process started at $(B,0)$ and does
not impose any condition on the exercise boundary.

\begin{lemma}\label{lem:smalltimeentry}
For every $M>0$, there are $p_M,\varepsilon_M>0$ such that
\begin{equation}\label{eq:smalltimeentry}
 \PP_{B,0}\left(0<Y_\varepsilon<\varepsilon,\quad
       \log\frac{S_\varepsilon}B<-(M+1)\varepsilon\right)\ge p_M
 \quad(0<\varepsilon<\varepsilon_M).
\end{equation}
\end{lemma}

\begin{proof}
Brownian scaling identifies the law of
$(\varepsilon^{-1}\log(S_{\varepsilon t}/B),
  \varepsilon^{-1}Y_{\varepsilon t})$, $0\le t\le1$, with that of
$(\widetilde X_t^\varepsilon,Z_t^\varepsilon)$ solving
\[
 \begin{aligned}
 dZ_t^\varepsilon&=(a-\kappa\varepsilon Z_t^\varepsilon) dt
                       +\sigma\sqrt{Z_t^\varepsilon} dW_t,\\
 d\widetilde X_t^\varepsilon&=\Bigl(\mu-\frac{\varepsilon Z_t^\varepsilon}2\Bigr) dt
                       +\sqrt{Z_t^\varepsilon} dB_t,
 \end{aligned}
 \qquad Z_0^\varepsilon=\widetilde X_0^\varepsilon=0,
 \qquad  d\langle B,W\rangle_t=\rho dt.
\]
Couple these equations with $ dZ_t=a dt+\sigma\sqrt{Z_t} dW_t$,
$Z_0=0$. One-dimensional comparison gives $0\le Z_t^\varepsilon\le Z_t$.
It can be obtained by Tanaka's formula: the local time at zero of
the difference vanishes for the square-root diffusion coefficient,
and the drift on the set $Z^\varepsilon>Z$ is nonpositive. The first
moments yield
\[
 \E|Z_t-Z_t^\varepsilon|
 =at-\frac{a(1-e^{-\kappa\varepsilon t})}{\kappa\varepsilon}
 \le C\varepsilon t^2,\qquad 0\le t\le1.
\]
Since $(\sqrt u-\sqrt v)^2\le|u-v|$, Ito's isometry implies
\[
 \widetilde X_1^\varepsilon\longrightarrow
 \widetilde X_1:=\mu+\int_0^1\sqrt{Z_t} dB_t
 \quad\hbox{in probability},\qquad
 Z_1^\varepsilon\longrightarrow Z_1\quad\hbox{in }L^1.
\]
The extra drift tends to zero in $L^1$ because
$\varepsilon\E\int_0^1Z_t^\varepsilon dt\le a\varepsilon/2$.

The elementary affine Laplace-transform calculation gives
\[
 \E e^{-\lambda Z_1}=(1+D\lambda)^{-\nu},\qquad\lambda\ge0.
\]
For example, apply Ito's formula to
$(1+D\lambda(1-t))^{-\nu}
 \exp\{-\lambda Z_t/(1+D\lambda(1-t))\}$.
Thus $Z_1$ has a gamma distribution with shape $\nu$ and scale $D$,
and $\PP(0<Z_1<1)>0$.
Write $B=\rho W+\sqrt{1-\rho^2}\,W^\perp$. Conditionally on $W$,
$\widetilde X_1$ is Gaussian with mean
$\mu+\rho(Z_1-a)/\sigma$ and variance
$(1-\rho^2)\int_0^1Z_t dt$. The latter is strictly positive on
$\{0<Z_1<1\}$. Therefore
\[
 \PP(\widetilde X_1<-(M+1),\ 0<Z_1<1)>0.
\]
Joint convergence and the open-set inequality for weak convergence
prove \eqref{eq:smalltimeentry}.
\end{proof}

If the ratio $d(v)/v$ stays bounded along a sequence, the process
enters the exercise region immediately and stock smooth fit holds.

\begin{lemma}\label{lem:entrysmoothfit}
Suppose that, for some $M>0$ and a sequence $v_n\searrow0$,
\begin{equation}\label{eq:boundedratioassumption}
d(v_n)\le Mv_n.
\end{equation}
Then the process started at $(B,0)$ enters the interior of the
exercise region at arbitrarily small positive times almost surely.
Moreover,
\begin{equation}\label{eq:conditionalstocksmoothfit}
 R_s(B+,0)=0,\qquad \lim_{v\searrow0}R_s(B,v)=0.
\end{equation}
\end{lemma}

\begin{proof}
Define
\[
 E_n=\{0<Y_{v_n}<v_n,\quad
          \log\frac{S_{v_n}}B<-(M+1)v_n\}.
\]
On $E_n$, monotonicity of $b$ and
\eqref{eq:boundedratioassumption} imply
\[
 S_{v_n}<Be^{-(M+1)v_n}<Be^{-Mv_n}
              \le b(v_n)\le b(Y_{v_n}).
\]
Thus the state is strictly inside the exercise region. Lemma
\ref{lem:smalltimeentry} gives $\liminf_n\PP(E_n)>0$.
The event $A=\limsup_n E_n$ belongs to the Brownian germ sigma-field:
for every $N$, it is measurable with respect to $\mathcal F_{v_N}$.
Also $\PP(A)\ge\limsup_n\PP(E_n)>0$. The Brownian zero--one law
therefore gives $\PP(A)=1$. No independence between the $E_n$ is
needed. To recall the zero--one argument, a germ event is independent
of Brownian increments after every positive time. Letting that time
decrease to zero makes it independent of the whole Brownian path,
and hence of itself.

Couple all initial stocks at zero variance by $S_t^{s,0}=sN_t$,
where $N_t=S_t^{1,0}$. For $0<h<K-B$, let $\tau_h$ be the first
contact time for the initial state $(B+h,0)$. This is optimal by
the Snell-envelope principle: the discounted reward is continuous,
bounded, and tends to zero at infinity. This standard principle
also allows $\tau_h=\infty$, with reward zero~\cite{PS}.
On a path in $A$, there are arbitrarily large $n$ for which $E_n$
occurs. For each such $n$ and all sufficiently small $h>0$,
\[
 (B+h)N_{v_n}<b(Y_{v_n}),\qquad \tau_h\le v_n.
\]
Consequently $\tau_h\to0$ almost surely as $h\searrow0$.

Put $\phi(s)=(K-s)^+$. Optimality for $B+h$ and admissibility of
the same stopping time for $B$ give
\begin{equation}\label{eq:smoothfitquotient}
 -1\le\frac{V(B+h,0)-V(B,0)}h
 \le\E\left[e^{-r\tau_h}
       \frac{\phi((B+h)N_{\tau_h})-\phi(BN_{\tau_h})}h\right].
\end{equation}
The integrand is defined as zero on $\{\tau_h=\infty\}$. It is
nonpositive and tends almost surely to $-1$, since $N_{\tau_h}\to1$
and $B<K$. Fatou's lemma applied to its negative shows that the
limsup of the right side is at most $-1$. This proves
$V_s(B+,0)=-1$, and hence the first assertion of
\eqref{eq:conditionalstocksmoothfit}.

For a fixed $\ell>0$, stock convexity gives
\[
 0\le R_s(B,v)
 \le\frac{R(B+\ell,v)-R(B,v)}\ell
 \le\frac{R(B+\ell,v)}\ell.
\]
The derivative exists for $v>0$ because $B>b(v)$. Send $v\searrow0$
using continuity of $R$, then send $\ell\searrow0$ using the stock
smooth fit just proved. This gives the second assertion.
\end{proof}

The smooth-fit conclusions of this lemma are obtained under
\eqref{eq:boundedratioassumption}. They are not used as prior
regularity assumptions on the model.

\begin{theorem}\label{thm:separation}
Under \eqref{eq:parameters},
\begin{equation}\label{eq:separation}
 \lim_{v\searrow0}\frac{d(v)}v=\infty.
\end{equation}
Consequently,
\begin{equation}\label{eq:endpointgeometry}
 \lim_{v\searrow0}\frac{b(v)-B}{v}=-\infty,
 \qquad \gamma'(0+)=0.
\end{equation}
\end{theorem}

\begin{proof}
If \eqref{eq:separation} failed, there would be a sequence satisfying
\eqref{eq:boundedratioassumption}. Convexity and $R(b(v),v)=0$ give
\[
 R(B,v)\le(B-b(v))R_s(B,v).
\]
Combine this with Lemmas~\ref{lem:contactnondegenerate} and
\ref{lem:entrysmoothfit}. Along the proposed sequence,
\[
 c_*v_n\le R(B,v_n)
 \le B(1-e^{-d(v_n)})R_s(B,v_n)
 \le BMv_n R_s(B,v_n)=o(v_n),
\]
a contradiction. Since $B-b(v)\sim Bd(v)$, the first conclusion
in \eqref{eq:endpointgeometry} follows. For the second, put
$\delta=d(v)$ and use $\gamma(\delta)/\delta=v/d(v)\to0$.
\end{proof}

These limits concern endpoint difference quotients. Continuity of
$\gamma'$ at zero and convergence of $b'(v)$ require a separate
argument, which is given in Theorem~\ref{thm:endpointregularity}.

We finally show that the tangent is continuous at the endpoint. The
anisotropic scale of Section~\ref{sec:axis} is useful for the boundary gap, but its
ellipticity constants are not uniform. To study the tangent, we instead
use an isotropic scale equal to the variance at a moving contact point.

\begin{theorem}\label{thm:endpointregularity}
For some $\delta_0>0$, the inverse boundary satisfies
\begin{equation}\label{eq:endpointc1}
 \gamma\in C^1([0,\delta_0)),\qquad
 \gamma'(0)=0,\qquad \gamma'(\delta)>0\quad(\delta>0).
\end{equation}
There is a nondecreasing function $\Omega$, depending on the model,
with $\Omega(s)\to0$ as $s\searrow0$, such that
\begin{equation}\label{eq:slopemodulus}
 0<\gamma'(\delta)\le
       \Omega\!\left(\frac{\gamma(\delta)}{\delta}\right)
 \quad\hbox{for all sufficiently small }\delta>0.
\end{equation}
In particular, $b'(v)\to-\infty$ as $v\searrow0$.
The modulus in \eqref{eq:slopemodulus} is not asserted to be a power.
\end{theorem}

\begin{proof}
Let $\varepsilon>0$, $\delta=d(\varepsilon)$,
$\eta=\varepsilon/\delta$, and $x_\varepsilon=x_0-\delta$.
Define
\[
 F_\varepsilon(X,Y)=
       \frac{w(x_\varepsilon+\varepsilon X,\varepsilon Y)}\varepsilon,
 \qquad p_*=(0,1).
\]
The point $p_*$ is on the rescaled free boundary. On every fixed
rectangle about $p_*$ contained in $Y>0$, small $\varepsilon$
ensures $x_\varepsilon+\varepsilon X<x_0$, because
$\eta\to0$ by Theorem~\ref{thm:separation}. Hence
$0\le F_\varepsilon\le CY$ there. Its operator is
$\mathcal L_\varepsilon$ from \eqref{eq:scaledpde}, with source
$f_\varepsilon(X)=f(x_\varepsilon+\varepsilon X)$.
These operators are uniformly elliptic on the rectangle, with
uniformly smooth coefficients and sources. Local obstacle estimates
give uniform $W^{2,p}$ and $C^{1,\alpha}$ bounds on smaller rectangles.
Also, stock convexity and \eqref{eq:contactbound} give
\begin{equation}\label{eq:smallmovingderivative}
 0\le F_{\varepsilon,X}=w_x
 \le\frac{C\varepsilon Y e^{x_\varepsilon+\varepsilon X}}
                 {B-e^{x_\varepsilon+\varepsilon X}}
 \le C\eta
\end{equation}
on a fixed neighborhood of $p_*$.

We record why boundary regularity can be used uniformly here.
After dividing the rescaled equation by $Y$, its principal matrix
is $A_0$ in \eqref{eq:constantprincipal}. The source is
\[
 Q_\varepsilon=
 \frac{f_\varepsilon-\Bigl(r-q-\frac{\varepsilon Y}2\Bigr)F_{\varepsilon,X}
 -(a-\kappa\varepsilon Y)F_{\varepsilon,Y}
 +r\varepsilon F_\varepsilon}{Y}.
\]
Choose the preceding $W^{2,p}$ estimate with $p=4$. On a fixed ball
$B_{r_1}(p_*)\subset\{1/2<Y<3/2\}$, it gives
$\|Q_\varepsilon\|_{C^{1/2}}\le M_1$, independently of
$\varepsilon$. Since $Q_\varepsilon(p_*)=f(x_\varepsilon)\to f_0>0$,
we may choose $0<r_0<r_1$ with $M_1r_0^{1/2}\le f_0/4$.
For all sufficiently small $\varepsilon$,
\[
 Q_\varepsilon\ge\frac{f_0}2\quad\hbox{on }B_{r_0}(p_*),\qquad
 \|F_\varepsilon\|_{C^{1,1}(B_{\frac{r_0}2}(p_*))}\le M_2,
\]
where the second bound follows from the Harnack and Schauder argument
in Proposition~\ref{prop:positiveregularity}.
Monotonicity gives the same rotated $1$-Lipschitz boundary chart on a
common cylinder. After the fixed linear change normalizing $A_0$,
the chart radius and Lipschitz constant remain uniform. The ellipticity
constants of $A_0$, the radius $r_0$, the bounds $M_1,M_2,f_0^{-1}$,
and this chart constant are the data in hypotheses H1--H2 and in the
Lipschitz-boundary argument of
\cite[Section 1.2, Theorem 3 and its proof]{Caffarelli}.
In \cite[Section 1.2]{Caffarelli} the estimates are stated in terms of
the distance to the rest of the boundary, the $C^{1,1}$ norm of the
solution, and the $C^{1/2}$ norm and positive lower bound of the
source, and \cite[Theorem 3]{Caffarelli} asserts that the modulus of
continuity of the derivative of the boundary graph does not depend on
the base point. The normal oscillation estimate and the subsequent
boundary Hessian estimate in that proof therefore give common moduli of
continuity.
These apply to the normals and to the Hessians taken from continuation
up to the boundary on a smaller common cylinder. All these cylinders
stay a fixed distance from $Y=0$.

Write $n_\varepsilon=(n_{\varepsilon,x},n_{\varepsilon,y})$ for the
inward unit normal at $p_*$. Both components are positive by
Proposition~\ref{prop:positiveregularity}. The rescaled form of
\eqref{eq:normalhessian} is
\[
 D^2F_\varepsilon(p_*)=
       \lambda_\varepsilon n_\varepsilon\otimes n_\varepsilon,
 \qquad
 \lambda_\varepsilon=
       \frac{f(x_\varepsilon)}{n_\varepsilon^T A_0 n_\varepsilon}.
\]
In particular $0<c_*\le\lambda_\varepsilon\le C$.
The common boundary modulus ensures that the inward normal segment
$p_*+t n_\varepsilon$, $0<t<t_0$, lies in continuation for a
fixed $t_0>0$. The Hessian modulus just obtained gives a
nondecreasing $\omega(t)\to0$, independent of $\varepsilon$, with
\[
 \left|D^2F_\varepsilon(p_*+t n_\varepsilon)
       -\lambda_\varepsilon n_\varepsilon\otimes n_\varepsilon
 \right|\le\omega(t),\qquad 0<t<t_0.
\]
Integrating along this segment from the zero gradient at contact
and using \eqref{eq:smallmovingderivative} yields
\[
 t\bigl(c_* n_{\varepsilon,x}-\omega(t)\bigr)
       \le F_{\varepsilon,X}(p_*+t n_\varepsilon)\le C\eta.
\]
Taking $t=\sqrt\eta$ gives
$n_{\varepsilon,x}\le C(\sqrt\eta+\omega(\sqrt\eta))\to0$.
Therefore $n_{\varepsilon,y}\to1$.

An isotropic rescaling preserves normals. The original boundary
written as $y=\gamma(x_0-x)$ has inward normal proportional to
$(\gamma'(\delta),1)$. Consequently
\[
 \gamma'(\delta)=\frac{n_{\varepsilon,x}}{n_{\varepsilon,y}}
       \le C\bigl(\sqrt\eta+\omega(\sqrt\eta)\bigr).
\]
This proves \eqref{eq:slopemodulus} with the displayed modulus,
after enlarging its constant. Together with
$\gamma'(0+)=0$ and analyticity at positive arguments, it proves
\eqref{eq:endpointc1}. Finally,
$\gamma'(d(v))=-b(v)/b'(v)$, with $b(v)\to B>0$, proves the
last assertion.
\end{proof}

\section{Boundary profiles and the axis equation}\label{sec:axis}

We begin with exact identities obtained by integrating the equation
along vertical lines. Choose $\delta_0,y_0>0$ small enough that the inverse boundary is
defined on $(0,\delta_0)$ and \eqref{eq:contactbound} holds below $y_0$.
For $x\in I=(x_0-\delta_0,x_0)$, write
\[
 g(x)=\gamma(x_0-x),\qquad g_h(x)=\min\{g(x),h\}.
\]
The contact set on a vertical line is $0\le y\le g(x)$. Set
\begin{equation}\label{eq:verticalremainder}
 \mathcal E=\kappa y w_y-\frac y2w_{xx}-\rho\sigma y w_{xy}
                -\Bigl(r-q-\frac y2\Bigr)w_x+rw.
\end{equation}
On compact subintervals of $I$, this is a locally integrable function
that vanishes for all sufficiently small $y$. All derivatives in the
following identities can therefore be interpreted as weak derivatives.

\begin{proposition}\label{prop:verticalidentities}
For almost every $x\in I$ and every $0<h<y_0$, one has
\begin{equation}\label{eq:exactflux}
 \frac{f(x)}a g_h(x)^\nu
 =h^\nu\left(\frac{f(x)}a-w_y(x,h)\right)
       +\frac1D\int_0^h t^{\nu-1}\mathcal E(x,t) dt.
\end{equation}
Define, for $0<t<h$,
\begin{equation}\label{eq:verticalkernel}
 K_\nu(h,t)=\frac{t^{\nu-1}}D\int_t^h s^{-\nu} ds,
 \qquad
 \Psi_\nu(g,h)=g+g^\nu\int_g^h t^{-\nu} dt
 \quad(0<g\le h),
\end{equation}
with $\Psi_\nu(0,h)=0$. Then
\begin{equation}\label{eq:exactvalue}
 w(x,h)=\frac{f(x)}a\bigl[h-\Psi_\nu(g_h(x),h)\bigr]
             +\mathcal R_h(x),\qquad
 \mathcal R_h(x)=\int_0^h K_\nu(h,t)\mathcal E(x,t) dt.
\end{equation}
Both identities also hold in distributions in $x$.
\end{proposition}

\begin{proof}
Fix $J\Subset I$. The minimum of $g$ on $J$ is positive. Below a
positive height smaller than this minimum, $w$ and its derivatives
vanish identically. Above that height, the local $W^{2,p}$ regularity
from Section~\ref{sec:prelim} applies across the contact set. In
particular, $w$ is $C^1$ there and its gradient vanishes at contact
points. The graph $y=g(x)$ has planar measure zero. Consequently,
\[
 Dyw_{yy}+aw_y=f(x)\chi_{\{y>g(x)\}}+\mathcal E
\]
holds almost everywhere on $J\times(0,y_0)$.
Sobolev slicing allows vertical integration for almost every $x\in J$.
Since $a=D\nu$, the left side multiplied by $y^{\nu-1}$ is
$D(y^\nu w_y)_y$. Integration from a height in the contact region
to $h$ gives \eqref{eq:exactflux}. If $h\le g(x)$, all terms except
the displayed multiples of $h^\nu$ vanish and the same formula holds.

Integrating once more, using $w=0$ in the lower contact interval,
and interchanging the integrals gives the kernel in
\eqref{eq:verticalkernel}. For the constant source above $g$, the
result is
\[
 \frac{f(x)}a\int_g^h\left[1-\Bigl(\frac gt\Bigr)^\nu\right] dt
 =\frac{f(x)}a\bigl[h-\Psi_\nu(g,h)\bigr].
\]
This proves \eqref{eq:exactvalue}. The same calculation integrated
against a smooth test function proves the distributional assertions.
An exhaustion of $I$ by compact subintervals completes the proof.
\end{proof}

The deficit in \eqref{eq:exactvalue} has an explicit form:
\begin{equation}\label{eq:psiexplicit}
 \Psi_\nu(g,h)=
 \begin{cases}
 \displaystyle\frac{g^\nu h^{1-\nu}-\nu g}{1-\nu},&\nu\ne1,\\[4pt]
 g\Bigl(1+\log\frac hg\Bigr),&\nu=1.
 \end{cases}
\end{equation}
In particular, when $g/h\to0$,
\begin{equation}\label{eq:psiasymptotic}
 \Psi_\nu(g,h)\sim
 \begin{cases}
 \frac{g^\nu h^{1-\nu}}{1-\nu},&0<\nu<1,\\
 g\log\frac hg,&\nu=1,\\
 \frac{\nu g}{\nu-1},&\nu>1.
 \end{cases}
\end{equation}
These are asymptotics of the explicit kernel deficit, not yet of the
exercise boundary.

We next estimate the remainder after averaging in log-stock. The
following estimate uses values of $w$ rather than bounds for its
second derivatives near zero variance. For a test function
$\varphi\in C_c^\infty((0,\infty))$, put
\[
 \varphi_\delta(x)=\delta^{-1}
             \varphi\Bigl(\frac{x_0-x}\delta\Bigr).
\]
For small $\delta$, its support lies in $I$.

\begin{proposition}\label{prop:averagedremainder}
For $0<h<\min\{y_0,1\}$ and all sufficiently small $\delta>0$,
\begin{equation}\label{eq:averagedremainder}
 \left|\int_I\varphi_\delta(x)\mathcal R_h(x) dx\right|
 \le C_\varphi\left(h^2+\frac{h^2}{\delta}
                              +\frac{h^3}{\delta^2}\right).
\end{equation}
The constant is independent of $h$, $\delta$, and the minimum contact
height on the support of $\varphi_\delta$.
\end{proposition}

\begin{proof}
Suppress the subscript $\nu$ on $K$ and define
\[
 A_h=\int_0^h K(h,t)w(x,t) dt,\quad
 B_h=\int_0^h tK(h,t)w(x,t) dt,\quad
 C_h=\int_0^h\partial_t(tK(h,t))w(x,t) dt.
\]
Integration by parts in $t$ in \eqref{eq:verticalremainder} gives
\begin{equation}\label{eq:remainderweak}
 \mathcal R_h=-\kappa C_h-\tfrac12\partial_{xx}B_h
       +\rho\sigma\partial_xC_h-(r-q)\partial_xA_h
       +\tfrac12\partial_xB_h+rA_h
\end{equation}
in distributions on $I$. There is no upper boundary term because
$K(h,h)=0$, and there is no lower term because $w=0$ near $t=0$
on each compact subinterval of $I$.

The kernel is nonnegative, with
\[
 \partial_tK(h,t)=-D^{-1}h^{1-\nu}t^{\nu-2}.
\]
Its first two moments are
\[
 \int_0^h tK(h,t) dt=\frac{h^2}{2D(\nu+1)},\qquad
 \int_0^h t^2K(h,t) dt=\frac{h^3}{3D(\nu+2)}.
\]
Moreover, $\int_0^h t|\partial_t(tK)| dt\le C h^2$.
The bound $0\le w(x,t)\le Ct$ therefore gives
\[
 |A_h|+|C_h|\le Ch^2,\qquad |B_h|\le Ch^3
\]
uniformly for $x\in I$. Move all $x$ derivatives in
\eqref{eq:remainderweak} onto $\varphi_\delta$. Its first two
derivative norms in $L^1$ have orders $\delta^{-1}$ and
$\delta^{-2}$. This proves \eqref{eq:averagedremainder}, since the term
$h^3/\delta$ is absorbed by $h^2/\delta$.
\end{proof}

The source in the continuation equation can be removed exactly.
This separates a smooth particular solution from the contribution
of the exercise set.

\begin{proposition}\label{prop:particular}
There are entire functions $A,C$ of $y$, with $A(0)=C(0)=0$,
such that
\begin{equation}\label{eq:particular}
 P(x,y)=A(y)+e^xC(y),\qquad \Lop P=f,
\end{equation}
and, locally uniformly in $x$,
\begin{equation}\label{eq:particularexpansion}
 P(x,y)=\frac{f(x)}a y+
 \frac{rK(\kappa+r)-qe^x(\kappa+q-\rho\sigma)}{2a(D+a)}y^2
 +O(y^3).
\end{equation}
For $Z=P-w$, one has, locally on $\{x<\log K,y>0\}$,
\begin{equation}\label{eq:potentialpde}
 \Lop Z=f(x)\chi_{\{w=0\}}
\end{equation}
both almost everywhere and in distributions.
\end{proposition}

\begin{proof}
The equations for $A$ and $C$ are
\begin{align*}
 DyA''+(a-\kappa y)A'-rA&=rK,\\
 DyC''+\bigl(a+(\rho\sigma-\kappa)y\bigr)C'-qC&=-q.
\end{align*}
Writing $A=\sum_{j\ge1}A_jy^j$ and
$C=\sum_{j\ge1}C_jy^j$ gives $A_1=rK/a$, $C_1=-q/a$ and
\[
 A_{j+1}=\frac{\kappa j+r}{(j+1)(Dj+a)}A_j,\qquad
 C_{j+1}=\frac{(\kappa-\rho\sigma)j+q}{(j+1)(Dj+a)}C_j.
\]
The multipliers are $O(1/j)$, so both series converge everywhere.
Direct substitution proves \eqref{eq:particular}, including when
$q=0$ or a multiplier vanishes. Their second coefficients give
\eqref{eq:particularexpansion}.
Local $W^{2,p}$ regularity of $w$ implies that there is no boundary
measure in its second derivatives. Its gradient and Hessian vanish
almost everywhere on its zero set, and $\Lop w=f$ on its positive
set. Subtracting proves \eqref{eq:potentialpde}.
\end{proof}

We now pass to the boundary scale. For $\varepsilon>0$, set
\begin{equation}\label{eq:anisotropic}
 \delta_\varepsilon=d(\varepsilon),\qquad
 \eta_\varepsilon=\frac{\varepsilon}{\delta_\varepsilon},\qquad
 U_\varepsilon(X,Y)=
       \frac{w(x_0+\delta_\varepsilon X,\varepsilon Y)}{\varepsilon}.
\end{equation}
Strict monotonicity ensures $\delta_\varepsilon>0$. Whenever the inverse
is defined, the rescaled boundary is
\begin{equation}\label{eq:rescaledinverse}
 g_\varepsilon(X)=\frac{\gamma(-\delta_\varepsilon X)}\varepsilon,
 \qquad X<0,
\end{equation}
and $g_\varepsilon(-1)=1$. Thus the boundary has a fixed reference point
in these coordinates.

An exact change of variables in the continuation equation gives
\begin{align}
 DY U_{\varepsilon,YY}+aU_{\varepsilon,Y}
 &+\eta_\varepsilon\bigl(\rho\sigma YU_{\varepsilon,XY}
                           +(r-q)U_{\varepsilon,X}\bigr)
   +\frac{\eta_\varepsilon^2Y}{2}U_{\varepsilon,XX}\notag\\
 &-\frac{\varepsilon\eta_\varepsilon Y}{2}U_{\varepsilon,X}
  -\kappa\varepsilon YU_{\varepsilon,Y}-r\varepsilon U_\varepsilon
 =rK-qB e^{\delta_\varepsilon X}.
 \label{eq:anisoexact}
\end{align}
Theorem~\ref{thm:separation} gives $\eta_\varepsilon\to0$. The
vertical integral identities provide the compactness argument below.
The upper bound $U_\varepsilon\le CY$ holds on $X\le0$ locally for small
$\varepsilon$. On $X>0$, the additional term in the bound is
$B(e^{\delta_\varepsilon X}-1)/\varepsilon$, which need not be bounded.
Moreover, the tangential ellipticity in \eqref{eq:anisoexact} degenerates.
The compactness argument of Proposition~\ref{prop:compact} therefore
does not apply to this family with uniform constants.

The averaged remainder estimate and Theorem~\ref{thm:separation}
now give compactness of the value and the inverse boundary. We do not
assume boundary ratio estimates or convergence of derivative traces.

\begin{theorem}\label{thm:anisocompact}
Under \eqref{eq:parameters}, let $\varepsilon_n\searrow0$, $\delta_n=d(\varepsilon_n)$, and define
\[
 G_n(t)=\frac{\gamma(\delta_n t)}{\varepsilon_n},\qquad
 U_n(t,Y)=\frac{w(x_0-\delta_n t,\varepsilon_nY)}{\varepsilon_n},
 \qquad t,Y>0.
\]
These functions are defined on every fixed compact set for large $n$.
There is a subsequence and a continuous, nondecreasing function
$G:(0,\infty)\to(0,\infty)$ such that, locally uniformly,
\begin{equation}\label{eq:anisoconvergence}
 G_n\longrightarrow G,\qquad
 U_n(t,Y)\longrightarrow
 \frac{f_0}{a}\int_{\min\{G(t),Y\}}^Y
                 \left[1-\left(\frac{G(t)}z\right)^\nu\right] dz.
\end{equation}
An integral with identical endpoints in this formula is zero.
The limit satisfies $G(1)=1$. If $m=\min\{\nu,1\}$, then $G^m$ is
concave and
\begin{equation}\label{eq:limitratios}
 \min\{1,t\}\le G(t)^m\le\max\{1,t\},\qquad t>0.
\end{equation}
In particular, the limit value and its $Y$ derivative vanish at
$Y=G(t)$, and the rescaled contact sets converge on compact subsets.
\end{theorem}

\begin{proof}
By Theorem~\ref{thm:separation}, $\varepsilon_n/\delta_n\to0$.
The functions $G_n$ are increasing and $G_n(1)=1$. Helly selection,
applied first to $\arctan G_n$ and then diagonally, gives an extended
nondecreasing limit $G_*$ at almost every $t>0$. At this stage, zero,
infinite values, and jumps are allowed.

Fix $Y>0$. Apply \eqref{eq:exactvalue} and
\eqref{eq:averagedremainder} with $h=\varepsilon_nY$ and
$\delta=\delta_n$, and divide by $\varepsilon_n$. The averaged
remainder is bounded by
\[
 C_{\varphi,Y}\left(\varepsilon_n+
       \frac{\varepsilon_n}{\delta_n}
       +\frac{\varepsilon_n^2}{\delta_n^2}\right)\longrightarrow0.
\]
Homogeneity of $\Psi_\nu$ in its two arguments and
$f(x_0-\delta_n t)\to f_0$ show that $U_n(\cdot,Y)$ converges in
distributions to
\begin{equation}\label{eq:extendedprofile}
 \frac{f_0}{a}\left[Y-
             \Psi_\nu(\min\{G_*(t),Y\},Y)\right].
\end{equation}
Dominated convergence applies since the profile in brackets lies in
$[0,Y]$.

Each $U_n(\cdot,Y)$ is nonincreasing and convex. Indeed, $R$ is
nondecreasing and convex in stock, so $w$ is convex in logarithmic
stock. On compact $t$ intervals these functions are also uniformly
Lipschitz. To see this, convexity gives
\[
 0\le R_s(s,y)\le\frac{R(B,y)-R(s,y)}{B-s}
                    \le\frac{Cy}{B-s},\qquad s<B.
\]
Consequently $|\partial_tU_n(t,Y)|\le CY/t$.
Thus the distributional limit has a continuous convex representative
$F_Y$, and the convergence to $F_Y$ is locally uniform in $t$.

Let $c_0=f_0/a$ and $Q_Y(t)=Y-F_Y(t)/c_0$. This is nonnegative,
nondecreasing, and concave. For $t<1$ one has $G_n(t)\le1$, and for $t>1$ one has
$G_n(t)\ge1$. The monotonicity of the vertical profile in its contact
height, followed by continuity of $F_Y$ at $1$, therefore gives
\[
 Q_Y(1)=\Psi_\nu(\min\{1,Y\},Y).
\]
Nonnegativity and concavity on the full interval $(0,\infty)$ imply
\begin{equation}\label{eq:concavesqueeze}
 \min\{1,t\}Q_Y(1)\le Q_Y(t)
                     \le\max\{1,t\}Q_Y(1).
\end{equation}
For example, the lower bound for $t<1$ follows by concavity between
$1$ and a point tending to zero. The upper bound for $t>1$ follows
from the same argument with the roles of $1$ and $t$ interchanged.
The other two inequalities follow from monotonicity.

For $Y>1$, put
\[
 c_\nu(Y)=
 \begin{cases}
 \frac{Y^{1-\nu}}{1-\nu},&\nu<1,\\
 \log Y,&\nu=1,\\
 \frac\nu{\nu-1},&\nu>1.
 \end{cases}
\]
For each fixed $0<g<\infty$,
$\Psi_\nu(\min\{g,Y\},Y)/c_\nu(Y)\to g^m$.
For $g=0$ the quotient is zero, and for $g=\infty$ it tends to
infinity. Divide \eqref{eq:concavesqueeze} by $c_\nu(Y)$ and send
$Y\to\infty$. In view of \eqref{eq:extendedprofile}, this excludes
zero and infinite values of $G_*$ almost everywhere.
The nonnegative concave functions $Q_Y/c_\nu(Y)$ are locally
bounded by \eqref{eq:concavesqueeze}. Their almost-everywhere limit
$G_*^m$ has a continuous concave representative $Q$. This also
follows from local compactness of bounded concave functions.
Set $G=Q^{1/m}$. The same inequalities give
\eqref{eq:limitratios}, and $Q(1)=1$.

Monotonicity of $G_n$ and continuity of $G$ upgrade almost-everywhere
convergence to locally uniform convergence. The formula for $F_Y$
now holds at every $t$. Finally, monotonicity of $U_n$ in $Y$ and
continuity of the limiting profile upgrade convergence for each
fixed $Y$ to local uniform convergence in $(t,Y)$, by a finite grid
argument. Differentiating the explicit limiting profile in $Y$
gives zero derivative at contact. Uniform convergence of $G_n$
proves the contact-set assertion directly. It is not inferred from
smallness of $U_n$ alone.
\end{proof}

\begin{corollary}\label{cor:power}
Set $m=\min\{\nu,1\}$. For every
$0<\alpha<m$, there are $C_\alpha,v_\alpha>0$ such that
\begin{equation}\label{eq:power}
 B-b(v)\le C_\alpha v^\alpha,\qquad 0<v<v_\alpha.
\end{equation}
If $\nu\ge1$, then
\begin{equation}\label{eq:logexponent}
 \lim_{v\searrow0}\frac{\log d(v)}{\log v}=1.
\end{equation}
\end{corollary}

\begin{proof}
Every sequence $\delta_n\searrow0$ can be written as
$\delta_n=d(\varepsilon_n)$ with $\varepsilon_n=\gamma(\delta_n)$.
Theorem~\ref{thm:anisocompact} and \eqref{eq:limitratios} give
\[
 \limsup_{\delta\searrow0}
       \frac{\gamma(2\delta)}{\gamma(\delta)}\le2^{\frac1m}.
\]
Fix $\alpha<m$. For small $\delta$ the ratio is at most
$2^{1/\alpha}$. Iteration on dyadic intervals, followed by
monotonicity between the dyadic points, yields
$\gamma(\delta)\ge c_\alpha\delta^{1/\alpha}$.
Inverting gives $d(v)\le C_\alpha v^\alpha$, and
$B-b(v)=B(1-e^{-d(v)})\le Bd(v)$ proves
\eqref{eq:power}. When $\nu\ge1$, combine this bound
for every $\alpha<1$ with $d(v)\ge c_*v$ from
Proposition~\ref{prop:lower}, and divide their logarithms by $\log v<0$.
\end{proof}

These estimates do not yet identify a slowly varying correction or
an asymptotic coefficient. The axis equation derived below provides
the additional information needed for
Theorem~\ref{thm:sharp}. The logarithmic exponent already suffices
for the following obstruction to a H\"older continuous tangent.

\begin{corollary}\label{cor:noholder}
If $\nu\ge1$, the inverse boundary does not belong to
$C^{1,\tau}([0,\delta_0))$ for any $0<\tau\le1$ and any $\delta_0>0$.
More generally, with $m=\min\{\nu,1\}$, such regularity would
require $\tau\le m^{-1}-1$.
\end{corollary}

\begin{proof}
Since $\gamma'(0)=0$, a $\tau$-H\"older derivative at zero would
give $\gamma(\delta)\le C\delta^{1+\tau}$.
The proof of Corollary~\ref{cor:power} gives
$\gamma(\delta)\ge c_\alpha\delta^{1/\alpha}$ for every $\alpha<m$.
If $1+\tau>1/m$, choose $\alpha<m$ with
$1/\alpha<1+\tau$. The two bounds contradict each other as
$\delta\searrow0$.
\end{proof}

\begin{proof}[Proof of Theorem~\ref{thm:main}]
Boundary positivity follows from Lemma~\ref{lem:positiveboundary}.
Propositions~\ref{prop:upper} and \ref{prop:strict} give continuity
and strict monotonicity. Theorem~\ref{thm:separation} gives scale
separation and the endpoint statements. Corollary~\ref{cor:power}
gives the power bounds and the logarithmic exponent.
Proposition~\ref{prop:positiveregularity} and
Theorem~\ref{thm:endpointregularity} give the differentiability
claims, and Corollary~\ref{cor:noholder} gives the obstruction
to a H\"older derivative at zero.
\end{proof}

We now derive an equation on the zero-variance axis. The
representation uses an explicit kernel with its pole on the axis. The outer data and the lower-order terms form a
$C^{1,\lambda}$ remainder for every $0<\lambda<1$.
We retain these terms throughout the argument.

\begin{lemma}\label{lem:axislinear}
There are $C,v_0>0$ such that
\[
 0\le V(e^x,v)-V(e^x,0)\le Cv,
 \qquad x\in\R,\quad 0<v<v_0.
\]
\end{lemma}
\begin{proof}
Put $U(x,v)=V(e^x,v)$. The function $(K-e^x)^+$ is globally
$K$-Lipschitz. Coupling log-stock processes with the same variance
and Brownian paths, and then taking the supremum over stopping times,
shows that $U(\cdot,v)$ has the same Lipschitz constant.
Start the process at $(x,0)$. Since $\E Y_t\le at$,
\[
 \E|X_h-x|\le |r-q|h+\frac12\int_0^h\E Y_t dt
                 +\left(\int_0^h\E Y_t dt\right)^{\frac12}\le Ch.
\]
The dynamic programming inequality and $U\le K$ imply
\[
 \E[U(x,Y_h)-U(x,0)]\le Ch.
\]
The integrand is nonnegative and increasing in $Y_h$. For $h=4v/a$,
the first two CIR moments and the Paley--Zygmund inequality give
$\PP(Y_h\ge v)\ge p_0>0$ for small $v$. Thus
$p_0[U(x,v)-U(x,0)]\le Ch$. All constants are independent of $x$.
\end{proof}

For the remainder of the paper set
\begin{equation}\label{eq:axisconstants}
 s_\rho=\sqrt{1-\rho^2},\quad \Lambda=\frac\sigma{s_\rho},\quad
 c=\frac{2(r-q)-\rho\nu\sigma}{\sigma s_\rho},\quad
 E=e^{-\pi c},\quad b_\rho=\frac\rho{s_\rho}.
\end{equation}
The coordinate $\xi=\Lambda(x-x_0)-b_\rho y$ removes the mixed
second derivative. Direct substitution gives
\[
 \frac{\Lop}{Dy}=\partial_{\xi\xi}+\partial_{yy}
 +\frac\nu y\partial_y+\frac c y\partial_\xi
 +\frac{2\rho\kappa-\sigma}{2Ds_\rho}\partial_\xi
 -\frac\kappa D\partial_y-\frac r{Dy}.
\]
Let $Z=P-w$, with $P$ as in Proposition~\ref{prop:particular}.
Lemma~\ref{lem:axislinear} and stock Lipschitz continuity imply, locally,
\begin{equation}\label{eq:axisZbounds}
 |Z(\xi,y)-Z(\xi,0)|\le Cy,\qquad |Z_\xi|\le C.
\end{equation}

\begin{proposition}\label{prop:axisequation}
Define
\[
 K_c(z,y)=(z^2+y^2)^{-\frac\nu2}e^{-c\arg(z+iy)},\qquad y>0,
 \quad h(t)=\gamma\Bigl(\frac t\Lambda\Bigr),\quad m(t)=f\Bigl(x_0-\frac t\Lambda\Bigr).
\]
For some $R>0$, the potential
\begin{equation}\label{eq:axisQ}
 Q(\delta)=\int_0^R m(t)\int_0^{h(t)}y^{\nu-1}
               K_c(\delta-t-b_\rho y,y) dy dt
\end{equation}
satisfies $Q=H$ for small $\delta>0$, where
$H\in C^{1,\lambda}$ for every $0<\lambda<1$ on an interval about
zero. On this whole interval, $Q-H$ is a positive constant times
$w(x_0-\delta/\Lambda,0)$. In particular, $Q-H$ is nonnegative and
convex, and vanishes on the positive half of the interval.
\end{proposition}
\begin{proof}
We first extend the equation to the axis. Write
\[
 \mathscr PZ=\operatorname{div}(y^\nu\nabla Z)
                         +cy^{\nu-1}Z_\xi.
\]
The equation for $Z$ takes the form
\begin{equation}\label{eq:weightedZ}
 \mathscr PZ=g+\operatorname{div}(y^\nu\mathbf B Z)
                         +d_0y^{\nu-1}Z,
\end{equation}
where
\[
 g=D^{-1}f y^{\nu-1}\chi_{\{w=0\}},\quad
 \mathbf B=\left(-\frac{2\rho\kappa-\sigma}{2Ds_\rho},\frac\kappa D\right),
 \qquad d_0=\frac{r-\nu\kappa}D.
\]
This identity extends to tests supported in the closed half-plane,
with no measure on the axis. To see this, take a smooth cutoff
$\chi_\varepsilon(y)$ equal to zero for $y\le\varepsilon$ and to
one for $y\ge2\varepsilon$, with
$|\chi_\varepsilon^{(j)}|\le C\varepsilon^{-j}$ for $j=1,2$.
Write $Z=Z(\xi,0)+R(\xi,y)$, where $|R|\le Cy$ by
\eqref{eq:axisZbounds}. For a smooth compactly supported test $\phi$,
the integral of $\partial_y(y^\nu\phi_y)$ in $y$ is zero, and the same
holds with $\phi$ replaced by $\chi_\varepsilon\phi$.
Thus the trace $Z(\xi,0)$ cancels from the vertical cutoff error.
The terms containing derivatives of the cutoff are bounded by
\[
 C_\phi\int_\varepsilon^{2\varepsilon}
 y\left(y^\nu\varepsilon^{-2}
       +y^{\nu-1}\varepsilon^{-1}
       +y^\nu\varepsilon^{-1}\right) dy
 \le C_\phi\varepsilon^\nu.
\]
Terms without a cutoff derivative converge by dominated convergence.
The first-order term on the right of \eqref{eq:weightedZ} contributes
at most $C_\phi\varepsilon^{-1}\int_\varepsilon^{2\varepsilon}y^\nu dy
=O(\varepsilon^\nu)$, and the remaining omitted terms have the same
bound. Since $\nu>0$, every cutoff error tends to zero. No pointwise
normal-flux trace is used.

We next test the equation against an explicit adjoint kernel. With
$R_z=(z^2+y^2)^{1/2}$, direct differentiation gives
\[
 (\log K_c)_z=\frac{-\nu z+cy}{R_z^2},\quad
 (\log K_c)_y=\frac{-\nu y-cz}{R_z^2},\quad
 \frac{\Delta K_c}{K_c}=\frac{\nu^2+c^2}{R_z^2}.
\]
Consequently $\mathscr P^*K_c=0$ away from its axis pole. Its flux
through a small semicircle is $-M_{\nu,c}$, where
\begin{equation}\label{eq:polemass}
 M_{\nu,c}=\frac{\nu^2+c^2}{\nu}
                    \int_0^\pi e^{-c\theta}\sin^\nu\theta d\theta>0.
\end{equation}
To compute this flux, integrate
$e^{-c\theta}(\nu\sin^\nu\theta+
c\sin^{\nu-1}\theta\cos\theta)$ and integrate the second term
by parts. The axis flux away from the pole is zero.

Choose a smooth cutoff $\chi$ equal to one near the pole $(s,0)$.
Testing \eqref{eq:weightedZ} against the kernel, first cut off at
distance $\varepsilon$ from the pole, yields
\begin{align*}
 M_{\nu,c}Z(s,0)={}&H_{\rm out}(s)-\int\chi K_c g\\
 &+\int y^\nu\mathbf B Z\cdot\nabla(\chi K_c)
                       -d_0\int\chi K_c y^{\nu-1}Z.
\end{align*}
Here $K_c$ has argument $(\xi-s,y)$, and $H_{\rm out}$ is smooth.
The error at the pole tends to zero: by \eqref{eq:axisZbounds},
$Z-Z(s,0)=O(|\xi-s|+y)$, whereas the cutoff operator has bounded
$L^1$ norm. All displayed integrals are locally integrable.

It remains to prove the regularity of the remainder. The two
lower-order potentials have a tangential derivative with
modulus at most $C\tau(1+|\log\tau|)$. To see this, move one
tangential derivative from the kernel onto $\chi Z$. The resulting
densities are bounded by \eqref{eq:axisZbounds}. The kernels are
bounded respectively by
\[
 y^\nu R_z^{-\nu-1},\qquad y^{\nu-1}R_z^{-\nu}.
\]
Their integrals over $R_z<\varepsilon$ are $O(\varepsilon)$, and
the integrals of their tangential derivatives over
$\varepsilon<R_z<1$ are $O(|\log\varepsilon|)$. Splitting a
translation difference at radius $2\tau$ proves the claimed modulus.
The angular integrals are finite for every $\nu>0$.

For the source integral keep the original horizontal variable:
$t=\Lambda(x_0-x)$ and $\xi=-t-b_\rho y$. The contact set is
exactly $0<y\le h(t)$, $t>0$, locally. No new graph after the shear
is needed. Contributions away from $t=0$ are smooth. The source
term is $D^{-1}Q$, while $Z(-\delta,0)=-w(x_0-\delta/\Lambda,0)$.
Absorbing the other terms into $H$ gives
$Q-H=DM_{\nu,c}w(x_0-\delta/\Lambda,0)$, proving the result.
\end{proof}

From now on write $m_0=m(0)=f_0$.
Theorem~\ref{thm:anisocompact} applies to $h$. In particular,
$h(t)/t\to0$, and every sequence $\delta_n\searrow0$ has a
subsequence for which $h(\delta_n t)/h(\delta_n)$ converges locally
uniformly. If $m_*=\min(\nu,1)$, its limit $G$ has $G(1)=1$ and
$G^{m_*}$ is nondecreasing and concave. For each $\epsilon>0$,
\begin{equation}\label{eq:dyadicinput}
 h(2t)^{m_*}\le2^{1+\epsilon}h(t)^{m_*},\qquad
 h(t)^{m_*}\ge c_\epsilon t^{1+\epsilon}
\end{equation}
for small $t$. These facts will be used before any regular variation
of $h$ has been proved.

\begin{lemma}\label{lem:quantprofile}
Put $\eta=h(\delta)/\delta$ and $G_\delta(t)=h(\delta t)/h(\delta)$.
On every compact subinterval of $(0,\infty)$,
\begin{equation}\label{eq:quantprofile}
 |G_\delta(t)-G_\delta(s)|\le C(|t-s|+\delta+\sqrt\eta).
\end{equation}
\end{lemma}
\begin{proof}
Let $\varepsilon=h(\delta)=\gamma(\delta/\Lambda)$, and choose
a fixed $Y$ larger than the uniform upper bound for $G_\delta$
on a slightly larger interval.
Set $\Phi(g)=\Psi_\nu(g,Y)$ and
\[
 F_\delta(t)=\Phi(G_\delta(t)),\qquad
 Q_\delta(t)=Y-\frac{a}{f_0\varepsilon}
                         w\Bigl(x_0-\frac{\delta t}\Lambda,\varepsilon Y\Bigr).
\]
The second function is uniformly bounded and concave, hence uniformly
Lipschitz on the smaller interval. The first is nondecreasing, and
$\Phi'$ is bounded below by a positive constant on the relevant range.
Equations \eqref{eq:exactvalue} and \eqref{eq:remainderweak}, with
the fixed factors of $\Lambda$ absorbed into the constant, give
\[
 |\langle F_\delta-Q_\delta,\phi\rangle|
 \le C\{(\delta+\varepsilon)\|\phi\|_1+
              \eta\|\phi'\|_1+\eta^2\|\phi''\|_1\}.
\]
Convolve with a positive mollifier of width $\tau$. Monotonicity
places $F_\delta(t)$ between its mollifications at $t-\tau$ and
$t+\tau$. Comparing these to $Q_\delta$ gives an error at most
$C(\tau+\delta+\varepsilon+\eta/\tau+\eta^2/\tau^2)$.
Take $\tau=\sqrt\eta$ and use the Lipschitz inverse of $\Phi$.
\end{proof}

\section{Sharp boundary asymptotics}\label{sec:rates}

We now derive the exact equivalents from the axis equation, starting
with the case $\nu>1$.

\begin{proposition}\label{prop:superrate}
If $\nu>1$, then
\[
d(v)\sim\Lambda^{-1}
 \left(\frac{\nu-1}{J_\nu(c)}\right)^{\frac1{\nu-1}}
 v\Bigl(\log\frac1v\Bigr)^{\frac1{\nu-1}},\qquad
 J_\nu(c)=\int_0^\pi e^{c\theta}\sin^{\nu-2}\theta d\theta.
\]
\end{proposition}
\begin{proof}
Use \eqref{eq:axisQ} and write $\eta(t)=h(t)/t$.
The derivative at $\delta=0$ of the inner integral is
\begin{equation}\label{eq:Dtail}
 D_c(t)=\frac E t\eta(t)^\nu(1+O(\eta(t))).
\end{equation}
This follows by expanding $K_c(-t-b_\rho y,y)$ for $y/t\to0$.
It is positive for small $t$.

We first justify differentiation under the integral. In the part
$t>2\delta$ the difference quotient is positive near the origin
and comparable to $D_c(t)$. The negative part for $t<2\delta$
is bounded by $C\int_0^{2\delta}\eta(t)^\nu dt=o(h(\delta))$.
Indeed, with $S=\sup_{t<2\delta}\eta(t)\to0$ and $H_*=h(2\delta)$,
\[
 \eta(t)^\nu\le\min(S^\nu,H_*^\nu t^{-\nu}),\qquad
 \int_0^{2\delta}\eta(t)^\nu dt\le C S^{\nu-1}H_*.
\]
The dyadic bound makes $H_*\le Ch(\delta)$. Fatou's lemma and
the finite derivative of $Q=H$ now imply $\int_0^R mD_c<\infty$.
The positive local contribution is $O(h(\delta))=o(\delta)$.
Dominated convergence on $t>2\delta$ then gives
$Q'(0)=\int_0^R m(t)D_c(t) dt$.

At the observation point rescale $t=\delta+h(\delta)u$ and
$y=h(\delta)z$. Compactness of the profiles implies
$h(\delta+h(\delta)u)/h(\delta)\to1$ uniformly for bounded $u$.
For the remaining local tail use monotonicity and
$h(2\delta)\le Ch(\delta)$. Since $\nu>1$, the resulting
$|u|^{-\nu}$ majorant is integrable. Translation by $b_\rho z$
does not change the integral over the whole $u$ axis. Hence
\[
 \int_0^{2\delta}m(t)\int_0^{h(t)}y^{\nu-1}
 K_c(\delta-t-b_\rho y,y) dy dt
 =m_0 C_\nu(c)h(\delta)+o(h(\delta)),
\]
where
$C_\nu(c)=\int_0^\pi e^{-c\theta}\sin^{\nu-2}\theta d\theta
 =EJ_\nu(c)$. This constant follows by putting $u=-z\cot\theta$
and then integrating $0<z<1$.

The local potential at zero is $o(h(\delta))$ by the preceding
bound with $S$. On $t>2\delta$, Taylor's remainder is bounded by
\[
 C\delta^2\int_{2\delta}^R\frac{h(t)^\nu}{t^{\nu+2}} dt
 \le C\frac{h(\delta)^\nu}{\delta^{\nu-1}}+O(\delta^2)
 =o(h(\delta)).
\]
For the inequality, iterate \eqref{eq:dyadicinput} with
$\nu\epsilon<1$. The same lower bound for $h$ shows that the
$C^{1,\lambda}$ Taylor remainder of $H$ is $o(h(\delta))$.
The interval $(\delta,2\delta)$ in the derivative tail contributes
$O(h(\delta)\eta(\delta)^{\nu-1})$. We obtain
\begin{equation}\label{eq:superselection}
 \int_0^\delta\eta(t)^\nu\frac{dt}{t}
                       \sim J_\nu(c)\eta(\delta).
\end{equation}
The source variation $m(t)-m_0=O(t)$ and the error in
\eqref{eq:Dtail} are negligible by the same bounds.

Put $\ell=\log(1/\delta)$, $e(\ell)=\eta(e^{-\ell})$, and
$F(\ell)=\int_\ell^\infty e(s)^\nu ds$. Then
$F\sim J_\nu(c)e$ and $F'=-e^\nu$, so
\[
 (F^{1-\nu})'\longrightarrow\frac{\nu-1}{J_\nu(c)^\nu}.
\]
Integrating and inverting yields
\[
 h(t)\sim t\biggl\{\frac{J_\nu(c)}{(\nu-1)\log\frac1t}\biggr\}^{\frac1{\nu-1}}.
\]
Finally $v=h(t)$ and $d(v)=t/\Lambda$, proving the formula.
\end{proof}

In the critical case $\nu=1$ we use the constants $C_c$ and $k_c$
defined in \eqref{eq:criticalconstants}.

\begin{proposition}\label{prop:criticalrate}
If $\nu=1$, then
\begin{equation}\label{eq:criticalrate}
d(v)\sim\Lambda^{-1}e^{1-k_c-\frac1{C_c}}
                 v\exp\sqrt{\frac2{C_c}\log\frac1v}.
\end{equation}
In particular, $c=0$ gives $C_c=2$, $k_c=\log2$ and coefficient
$\sqrt e/(2\Lambda)$.
\end{proposition}
\begin{proof}
We first derive the leading balance. Again put $\eta(t)=h(t)/t$, $\ell=\log(1/\delta)$ and
$e(\ell)=\eta(e^{-\ell})$, but now set
$F(\ell)=\int_\ell^\infty e(s) ds$ and $H_\ell=\log(1/e(\ell))$.
The Fatou argument used above remains valid. The negative local
part divided by $\delta$ tends to zero because $\eta(t)\to0$.
The positive local part has size $h(\delta)\log(1/\eta(\delta))
=o(\delta)$. Thus the derivative formula and $F<\infty$ hold.

We first establish a leading relation without assuming regular
variation. Compactness and the ratio bounds imply that the local
positive contribution $L_\delta$ satisfies
\[
 L_\delta\sim m_0(1+E)h(\delta)H_\ell.
\]
For completeness, remove $|t/\delta-1|<M\eta$, whose contribution
is $O(h)$, and $|t/\delta-1|>\zeta$, whose local contribution is
also $O(h)$. On the intervening intervals, the heights are between
$h(\delta)(1-o_\zeta(1))$ and $h(\delta)(1+o_\zeta(1))$.
The two logarithmic coefficients are $1$ and $E$. Send
$\delta\searrow0$, then $\zeta\searrow0$.

The local potential $N_\delta$ at zero is nonnegative. It must not
yet be bounded by $O(h)$. The Taylor equation instead gives
$m_0E\delta F(1+o(1))=L_\delta-N_\delta+O(h)$, hence
$F=O(eH_\ell)$. Monotonicity of $h$ gives, for every fixed $M>0$,
\[
 \int_0^\infty e^{-u}e(\ell+u) du
 \le(e^M-1)e(\ell)+e^{-M}F(\ell).
\]
This bounds $N_\delta/(hH_\ell)$ by $C_M/H_\ell+Ce^{-M}$.
It follows that
\begin{equation}\label{eq:criticalleading}
 F\sim C_c eH_\ell,\qquad H_\ell\sim\sqrt{\frac{2\ell}{C_c}}.
\end{equation}
For the second relation take $T=\log(1/F)$. The first gives
$T\sim H_\ell$ and $T'\sim1/(C_cT)$. Moreover $F'/F\to0$,
so $F(\ell+u)/F(\ell)\to1$ for fixed $u$. The same follows for
$e$ from \eqref{eq:criticalleading}. Thus $h$ is regularly varying
with index one. In particular, its ratios obey Potter bounds with
exponents $1\pm\epsilon$.

To determine the coefficient, one more term of the local integral is
necessary. Suppress
$m$ temporarily and put
\[
 F_\delta(t)=\int_0^{h(t)}K_c(\delta-t-b_\rho y,y) dy,
 \qquad D_c(t)=\partial_\delta F_\delta(t)|_{\delta=0}.
\]
Consider the convergent combination
\begin{equation}\label{eq:criticalB}
 \mathcal B_\delta=\int_0^R
 [F_\delta(t)-F_0(t)-\delta\chi_{\{t>\delta\}}D_c(t)] dt.
\end{equation}
Replacing $h(\delta u)/h(\delta)$ by $u$ in this expression has
error $o(h(\delta))$. The following bounds justify this replacement
at the singular point. Away from $u=1$, the normalized integrand is
bounded by $CG_\delta(u)/u$ for $u<1/2$ and
$CG_\delta(u)/u^3$ for $u>2$. Potter bounds make these integrable.
For $M\eta<|u-1|<\zeta$, Lemma~\ref{lem:quantprofile} bounds the
replacement error, divided by $h$, by
\[
 C\int_{M\eta<|u-1|<\zeta}
      \frac{|G_\delta(u)-u|}{|u-1|} du
 \le C\zeta+C(\delta+\sqrt\eta)\log\frac1\eta.
\]
The last term tends to zero by \eqref{eq:dyadicinput}. In the central
interval, put $H=h(\delta)$, $\eta=H/\delta$, and set
$t=\delta+Hz$, $y=Hs$. The two rescaled heights are
\[
 a_\delta(z)=G_\delta(1+\eta z),\qquad
 b_\delta(z)=1+\eta z,\qquad |z|\le M.
\]
They tend uniformly to one, so the interval between them lies in
$[1/2,2]$ for small $\delta$. Homogeneity of $K_c$ when $\nu=1$
expresses the difference of the central $F_\delta$ integrals as
\[
 H\int_{-M}^M\int_{b_\delta(z)}^{a_\delta(z)}
       K_c(-z-b_\rho s,s) ds dz.
\]
Its absolute value is at most
$C_MH\sup_{|z|\le M}|a_\delta(z)-b_\delta(z)|=o(H)$.
On the same interval, the kernels in the subtracted $F_0$ and
$\delta D_c$ terms are bounded by $C_M/\delta$, so their replacement
errors are $O(H^2/\delta)=o(H)$.
These estimates allow
$\delta\searrow0$, followed by $\zeta\searrow0$ and removal
of the integrable outer tails.

For the linear profile the inner integral, with symmetric cutoff $U$,
is
\[
 \int_{-U}^U\int_0^1
 (u^2+z^2)^{-\frac12}e^{-c\arg(-u+iz)} dz du
 =(1+E)\log U+(1+E)+A_c+o(1).
\]
This is the substitution $u=-z\cot\theta$. The shear translates
the $u$ interval by the bounded amount $b_\rho z$ and does not
alter its finite part. The outer integrals for the linear profile are
\[
 \int_0^{1-\zeta}\left(\frac u{1-u}-E\right) du
 +E\int_{1+\zeta}^\infty
       \left(\frac u{u-1}-1-\frac1u\right) du
 =-(1+E)\log\zeta-(1+E)+o(1).
\]
Combining them with $U=\zeta/\eta$ proves
\[
 \frac{\mathcal B_\delta}{h(\delta)}
             =(1+E)H_\ell+A_c+o(1).
\]
Restoring $m$ multiplies this by $m_0$. The local error is
$O(\delta H_\ell)=o(1)$, and the outer errors follow from the
same Taylor and Potter bounds. Also
\[
 D_c(t)=\frac{E\eta(t)}t+O\Bigl(\frac{\eta(t)^2}t\Bigr),\qquad
 \int_\ell^\infty e(s)^2 ds
 \le\sup_{s\ge\ell}e(s)F(\ell)=o(e(\ell)).
\]
The last equality follows from \eqref{eq:criticalleading}.
The Taylor remainder of the axis function is $o(h)$ by
\eqref{eq:dyadicinput}. Therefore \eqref{eq:criticalB} yields
\begin{equation}\label{eq:criticalnext}
 F=C_c e(H_\ell+k_c+o(1)).
\end{equation}

It remains to integrate and invert this relation. We integrate it
without differentiating its error term.
With $T=\log(1/F)$,
\[
 H_\ell=T+\log T+\log C_c+o(1),\qquad
 T'=\{C_c(T+\log T+\log C_c+k_c+o(1))\}^{-1}.
\]
Integrating with $T$ as the independent variable and substituting
back gives
\[
 \ell=C_c\Bigl\{\frac{H_\ell^2}2+(k_c-1)H_\ell+o(H_\ell)\Bigr\},\qquad
 e(\ell)\sim e^{k_c-1}e^{-\sqrt{\frac{2\ell}{C_c}}}.
\]
Finally, $v=h(t)$ implies
\[
 \log\frac1v=\ell+\sqrt{\frac{2\ell}{C_c}}+1-k_c+o(1).
\]
Consequently
\[
 \sqrt{\frac{2\ell}{C_c}}=\sqrt{\frac2{C_c}\log\frac1v}-\frac1{C_c}+o(1).
\]
Since $d(v)=t/\Lambda$, this proves \eqref{eq:criticalrate}.
\end{proof}

We finally turn to the power regime. From here until the end of the
proof of Proposition~\ref{prop:subcriticalrate}, we assume $0<\nu<1$. Put $\alpha=1-\nu$ and
retain $E=e^{-\pi c}$. There is a unique
$\beta\in(\nu,1)$ such that
\begin{equation}\label{eq:beta}
 \sin\pi(\beta-\nu)=E\sin\pi\beta.
\end{equation}
Equivalently,
\[
 \beta=\frac{1+\nu}{2}-\frac1\pi
 \arctan\Bigl(\tanh\frac{\pi c}2\cot\frac{\pi\nu}2\Bigr).
\]
We first identify the boundary ratios and then prove convergence of
their amplitude. The second conclusion does not follow from the first.

\begin{proposition}\label{prop:subcriticalratios}
For every $t>0$,
\begin{equation}\label{eq:subcriticalratios}
 \frac{h(\delta t)}{h(\delta)}\longrightarrow t^{\frac\beta\nu}
 \qquad(\delta\searrow0),
\end{equation}
locally uniformly for $t\in(0,\infty)$.
\end{proposition}
\begin{proof}
We first identify the limiting potential. Along any subsequence from Theorem~\ref{thm:anisocompact}, set
\[
 M_\delta(t)=\left(\frac{h(\delta t)}{h(\delta)}\right)^\nu
                \longrightarrow M(t).
\]
The limit is nondecreasing and concave, $M(1)=1$, and
$\min(1,t)\le M(t)\le\max(1,t)$.
Define the homogeneous kernel
\[
 k(z)=\begin{cases}z^{-\nu},&z>0,\\E|z|^{-\nu},&z<0.\end{cases}
\]
Fix $0<a<b$ and subtract from $Q(\delta z)$ its secant through
$z=a,b$. Divide the result by
$N_\delta=\delta^{1-\nu}h(\delta)^\nu$.
After a fixed positive normalization this converges locally uniformly
on the entire real axis to
\begin{equation}\label{eq:renormpotential}
 \mathcal U(z)=\int_0^\infty
 \left[k(z-t)-\frac{b-z}{b-a}k(a-t)
                 -\frac{z-a}{b-a}k(b-t)\right]M(t) dt.
\end{equation}
The integral is understood with the bracket grouped at infinity.
We justify this limit without assuming differentiability at
zero. Locally, the finite-thickness kernel is bounded by
$C|z-t|^{-\nu}$, including the fixed shear. Since $\nu<1$, omitting
$|z-t|<\zeta$ gives an error of at most $C\zeta^{1-\nu}$ on bounded sets.
On the complement, $h(\delta t)/\delta\to0$ and dominated
convergence applies. At infinity the secant removes the constant
and linear terms, leaving a bound $CM_\delta(t)t^{-\nu-2}$.
The dyadic bound \eqref{eq:dyadicinput} gives
$M_\delta(t)\le Ct^{1+\epsilon}$ for $t>1$ up to a fixed original
radius, with $\epsilon<\nu$. Beyond that radius the remainder is
$O(\delta^2)/N_\delta\to0$. Thus the tails are uniformly integrable.
Finally the secant remainder of the smooth axis function, divided
by $N_\delta$, is at most
$C\delta^{\lambda+\nu-1-\epsilon}\to0$ on choosing
$\lambda>1-\nu$ and small $\epsilon$.

Proposition~\ref{prop:axisequation} now implies that $\mathcal U$
is nonnegative and convex on $\R$, and zero on $(0,\infty)$.
In particular, the second derivative of its potential is zero on
this half-line. These assertions are preserved under locally uniform
limits of admissible $M$: the same local and tail bounds apply.

Next we derive growth and integrability properties of the limit.
There is a uniform strict improvement of linear growth. Otherwise
a compact limit would satisfy $M(2)=2$. Concavity and $M(1)=1$
would then give $M(t)=t$ on $(0,2)$. For $z\searrow0$ the second
derivative of the left part of its potential would be
\[
 \frac{d^2}{dz^2}\int_0^z(z-t)^{-\nu}t dt=z^{-\nu}.
\]
The corresponding derivative of the integral from $z$ to $2$, and
of the twice normalized tail from $2$ to infinity, stays bounded.
This contradicts $\mathcal U''=0$ on $(0,\infty)$.
Compactness therefore gives $M(2)\le2-\kappa_0$ uniformly over all
limits. Returning to the original ratios and iterating, there is
$\epsilon_0\in(0,\alpha)$ such that
\begin{equation}\label{eq:strictgrowth}
 h(2t)^\nu\le2^{1-\epsilon_0}h(t)^\nu
 \quad\hbox{for small }t,\qquad
 M(t)\le Ct^{1-\epsilon_0}\quad(t\ge1).
\end{equation}

Convexity gives the required integrability at the other endpoint.
For $z<0$, differentiating \eqref{eq:renormpotential} is legitimate
away from zero. The only part of $\mathcal U'(z)$ that could diverge
as $z\nearrow0$ is
$\nu E\int_0^1 M(t)(t-z)^{-\nu-1} dt$, with positive sign.
But $\mathcal U\ge0$, $\mathcal U(0)=0$, and convexity give
$\mathcal U'(z)\le0$. Monotone convergence proves
\begin{equation}\label{eq:profiledini}
 \int_0^1 M(t)t^{-\nu-1} dt<\infty,
 \qquad M(0+)=0,\qquad t^{-\nu}M(t)\longrightarrow0.
\end{equation}
The last assertion follows by integrating on $(t,2t)$.
The derivative at zero can now be passed under the normalized
integral. The local right-side error is bounded by
$C[z^{-\nu}M(2z)+\int_0^{2z}M(t)t^{-\nu-1} dt]$, which tends
to zero. Thus $\mathcal U'(0-)=\mathcal U'(0+)=0$.
Define $W(z)=-\mathcal U'(-z)$ for $z>0$ and $W=0$ for $z\le0$.
Then
\begin{equation}\label{eq:stableW}
 W(z)=\nu E\int_0^\infty M(t)
       [t^{-\nu-1}-(t+z)^{-\nu-1}] dt\ge0.
\end{equation}
It is continuous and, by \eqref{eq:strictgrowth}--\eqref{eq:profiledini},
$W(z)\le C(1+z^{\alpha-\epsilon_0})$.

Finally, we classify the limit by a stable operator. We use the
half-space Liouville theorem for stable operators
\cite[Theorem 4.11]{DRSV}. With Fourier convention $e^{-i\xi x}$,
\[
 \widehat k(\xi)=\Gamma(\alpha)|\xi|^{-\alpha}
 \left[(1+E)\cos\frac{\pi\alpha}2
       -i(1-E)\operatorname{sgn}\xi\sin\frac{\pi\alpha}2\right].
\]
The inverse multiplier is the symbol of an operator $L$ of order
$\alpha$ with positive stable kernels in both directions. Indeed,
its skew ratio divided by $\tan(\pi\alpha/2)$ has absolute value
$|1-E|/(1+E)<1$. After reflection, the operator $\widetilde L$
annihilates $W$ on $(0,\infty)$ in the weak sense.

We give the cutoff argument for this last assertion to avoid applying
$L$ directly to the faster-growing function $\mathcal U$. Let
$M_R(t)=M(t)\chi(t/R)$ with a smooth compact cutoff, and put
$P_R=k*M_R$ and $W_R(z)=P_R'(0)-P_R'(-z)$. Concavity and
\eqref{eq:profiledini} imply
\[
 \int_0^1t^{-\nu}M'(t) dt<\infty,\qquad
 |M_R'(t)|\le Ct^{-\epsilon_0}\quad(t\ge1).
\]
The Fourier identity gives $\widetilde L W_R=-M_R'(-z)$.
Moving a derivative onto $M_R$ and splitting at $t=|z|/2,2|z|$
gives $|W_R(z)|\le C(1+|z|^{\alpha-\epsilon_0})$ uniformly in $R$.
Near $t=|z|$ use the integrable kernel $|t-|z||^{-\nu}$. In the
tail use $C|z|t^{-\nu-1}$, and near zero use the displayed weighted
integral. Thus $W_R$ converges locally and in the weighted tail
$L^1((1+|z|)^{-1-\alpha} dz)$ to $W$. Passing to the weak
equation proves the assertion, including the absence of any
polynomial ambiguity.

The growth exponent $\alpha-\epsilon_0$ is strictly less than the
order $\alpha=2s$ in the notation of \cite[Theorem 4.11]{DRSV}.
Together with $W=0$ on the negative half-line, the weak equation and
the positive kernels verified above give that theorem's hypotheses.
The theorem gives $W(z)=a_*z_+^\chi$, where $a_*>0$ by
\eqref{eq:stableW} and $M(1)=1$, and
\[
 \chi=\frac\alpha2-\frac1\pi
 \arctan\left(\frac{1-E}{1+E}\tan\frac{\pi\alpha}2\right)
 \in(0,\alpha).
\]
On the negative half-line, applying $\widetilde L$ to this power
gives $-M'(-z)$. Hence $M'(t)$ is a positive constant times
$t^{\chi-\alpha}$. The conditions $M(0)=0$, $M(1)=1$ imply
$M(t)=t^{1+\chi-\alpha}=t^\beta$, with $\beta=\nu+\chi$ as in
\eqref{eq:beta}. Every subsequential limit is therefore the same,
which proves \eqref{eq:subcriticalratios}.
\end{proof}

The convergence of the amplitude uses a Mellin argument, which the
next lemma states separately. We write
$B(z,w)=\Gamma(z)\Gamma(w)/\Gamma(z+w)$ for the Euler beta
function, continued meromorphically.

\begin{lemma}\label{lem:mellinamplitude}
Let $M$ be nonnegative and nondecreasing near zero, regularly varying
there with index $\beta$ from \eqref{eq:beta}, and bounded with compact
support in $[0,R]$. Set
\[
 P_M(x)=\int_0^R k(x-t)M(t) dt,\qquad
 r_M(x)=P_M(x)-P_M(0)-xP_M'(0),\quad x>0.
\]
Suppose that, for some $\kappa_*>0$,
\begin{equation}\label{eq:amplituderemainder}
 |r_M(x)|\le Cx^{\beta+\alpha+\kappa_*}\quad(0<x<x_1).
\end{equation}
Then $M(t)/t^\beta\to A\in(0,\infty)$, where
\begin{equation}\label{eq:amplitudeformula}
 A=-\frac1{\Theta'(\beta)}
             \int_0^\infty r_M(x)x^{-\beta-\alpha-1} dx,
 \qquad
 \Theta(s)=B(s+1,\alpha)+E B(-s-\alpha,\alpha).
\end{equation}
\end{lemma}
\begin{proof}
We first compute the Mellin multiplier. Since $\beta>\nu$, $P_M(0)$ and $P_M'(0)$ are finite. For
$\nu<\operatorname{Re}s<\beta$, Fubini and $x=tu$ give
\begin{equation}\label{eq:mellinproduct}
 \mathcal R(s):=\int_0^\infty r_M(x)x^{-s-\alpha-1} dx
 =\Theta(s)\widehat M(s),\qquad
 \widehat M(s)=\int_0^R M(t)t^{-s-1} dt.
\end{equation}
Indeed, the multiplier is the convergent integral
\[
 \int_0^\infty u^{-s-\alpha-1}
              [k(u-1)-E-\nu E u] du.
\]
The bracket is $O(u^2)$ at zero, $O(u)$ at infinity, and has an
integrable singularity at one. The integral is analytic for
$\nu<\operatorname{Re}s<\nu+1$. Evaluating its two parts by the
beta integral, with the displayed subtractions at zero, gives
\[
 \Theta(s)=B(s+1,\alpha)
       \left(1-E\frac{\sin\pi s}{\sin\pi(s-\nu)}\right).
\]
Its only zero in this strip is the simple real zero $\beta$.
To check that there are no complex zeros, solve
$\cot\pi s=(\cos\pi\nu-E)/\sin\pi\nu$. All its solutions
are real and differ by integers. The beta prefactor is nonzero
in the strip.

We next continue the transform and compute the leading residue. At
infinity $r_M(x)=O(1+x)$. Thus
\eqref{eq:amplituderemainder} makes $\mathcal R$ analytic for
$\nu<\operatorname{Re}s<\beta+\kappa_*$. Reduce $\kappa_*$ if
necessary so the upper endpoint is below $\nu+1$.
Equation \eqref{eq:mellinproduct} continues $\widehat M$
meromorphically across $\beta$, with at most a simple pole there.
On closed vertical substrips, $\mathcal R$ is bounded, while
Stirling's formula gives
$|\Theta(s)|\asymp(1+|\operatorname{Im}s|)^{-\alpha}$ for
large imaginary part. The two leading coefficients are proportional
to $e^{-i\pi\alpha/2}+Ee^{i\pi\alpha/2}$ or its conjugate,
which cannot vanish. The continued transform has polynomial growth.

Choose a nonnegative smooth $\psi$ supported in
$(e^{-\epsilon},e^\epsilon)$, with $\int\psi(u) du/u=1$, and
put $M_\psi(x)=\int M(xu)\psi(u) du/u$. Its Mellin multiplier
decays faster than any power on vertical lines. Moving the inverse
Mellin contour to the right of $\beta$ therefore gives, for
$0<\kappa'<\kappa_*$,
\[
 M_\psi(x)=Ax^\beta\int u^\beta\psi(u)\frac{du}u
                                      +O_\psi(x^{\beta+\kappa'}).
\]
The residue has the sign in \eqref{eq:amplitudeformula}, since
the transform of $At^\beta$ near zero has principal part
$A/(\beta-s)$. Monotonicity gives
$M(xe^{-\epsilon})\le M_\psi(x)\le M(xe^\epsilon)$ for small $x$.
Apply the expansion at $xe^{\pm\epsilon}$ and then send
$\epsilon\searrow0$. This proves $M(x)/x^\beta\to A\ge0$.
If $A=0$, one fixed smoothing and the same inequalities give
$M(x)=O(x^{\beta+\kappa'})$, contrary to regular variation of
index $\beta$. Thus $A>0$.
\end{proof}

\begin{proposition}\label{prop:subcriticalrate}
For $0<\nu<1$ there is $A\in(0,\infty)$ such that
\begin{equation}\label{eq:subcriticalrate}
d(v)\sim\Lambda^{-1}A^{-\frac1\beta}v^{\frac\nu\beta}.
\end{equation}
Formula \eqref{eq:amplitudeformula} determines $A$ using
\begin{equation}\label{eq:actualM}
 M(t)=\frac{m(t)}{m_0}h(t)^\nu\quad(0<t<R),\qquad M(t)=0\quad(t\ge R).
\end{equation}
Its value is independent of the sufficiently small cutoff $R$.
\end{proposition}
\begin{proof}
Proposition~\ref{prop:subcriticalratios} gives regular variation of
$h$ with index $p=\beta/\nu>1$. The function in
\eqref{eq:actualM} is nondecreasing near zero because both $h$ and
$m(t)=rK-qe^{x_0-t/\Lambda}$ are nondecreasing. It has index $\beta$.
Normalize \eqref{eq:axisQ} by $\nu/m_0$, and denote the resulting
potential by $\overline Q$. Its thin-layer approximation is $P_M$.
Write $T=\overline Q-P_M$.

If $|x-t|$ is larger than a sufficiently large multiple of $h(t)$,
Taylor expansion in $y/|x-t|$, including the fixed shear, gives
\begin{equation}\label{eq:thicknesskernel}
 \left|\int_0^{h(t)}y^{\nu-1}
 [K_c(x-t-b_\rho y,y)-k(x-t)] dy\right|
 \le C\frac{h(t)^{\nu+1}}{|x-t|^{\nu+1}}.
\end{equation}
After $j$ derivatives in $x$ the denominator has power $\nu+1+j$.
In particular $T(0),T'(0)$ are finite, because
$\int_0^R h(t)^{\nu+1}t^{-\nu-2} dt<\infty$ when $p>1$.

Split the Taylor remainder of $T$ into $t<x/2$, $x/2<t<2x$,
and $t>2x$. In the middle interval, $h(t)\asymp h(x)$.
On $|t-x|\le Ch(x)$ both the thick and thin potentials have
integral $O(h(x))$, since $\nu<1$. On the remainder of this interval,
\eqref{eq:thicknesskernel} integrates to $O(h(x))$ as well.
The subtracted values at zero on this interval are bounded by
$Ch(x)^{\nu+1}x^{-\nu}$. The first interval contributes at most
the same bound, including its zero and first-order Taylor terms. For
this, use Potter bounds in the integrals
$\int_0^x h(t)^{\nu+1}t^{-\nu-1} dt$ and
$x\int_0^x h(t)^{\nu+1}t^{-\nu-2} dt$.
The last interval contributes at most
$Cx^2\int_{2x}^R h(t)^{\nu+1}t^{-\nu-3} dt$.
Consequently, for arbitrarily small $\epsilon>0$,
\begin{equation}\label{eq:thicknessremainder}
 |T(x)-T(0)-xT'(0)|\le C_\epsilon
 \left(x^{p-\epsilon}+
           x^{\min\{p(\nu+1)-\nu,2\}-\epsilon}\right).
\end{equation}
Both exponents exceed $\beta+\alpha$: the respective positive gaps
before subtracting $\epsilon$ are
$(p-1)(1-\nu)$, $p-1$, and $1+\nu-\beta$.

For all sufficiently small $x>0$, chosen with $2x<R$,
$\overline Q$ equals a $C^{1,\lambda}$ function for every
$\lambda<1$. Choose $\lambda>\beta-\nu$. Its Taylor remainder
has order $x^{1+\lambda}$, again strictly higher than
$x^{\beta+\alpha}$. Since $P_M=\overline Q-T$,
\eqref{eq:thicknessremainder} proves the hypothesis of
Lemma~\ref{lem:mellinamplitude}. Hence
$h(t)\sim A^{1/\nu}t^{\beta/\nu}$. Inverting with $d(v)=t/\Lambda$
proves \eqref{eq:subcriticalrate}.

Changing $R$ adds to $M$ a function supported away from zero.
Its Mellin transform is analytic near $\beta$. By
\eqref{eq:mellinproduct} its contribution to $\mathcal R(\beta)$
is zero because $\Theta(\beta)=0$. This also verifies cutoff
independence directly in \eqref{eq:amplitudeformula}.
\end{proof}

\begin{proof}[Proof of Theorem~\ref{thm:sharp}]
Propositions~\ref{prop:superrate}, \ref{prop:criticalrate}, and
\ref{prop:subcriticalrate} give the three formulas. Continuity gives
$B-b(v)=B(1-e^{-d(v)})\sim Bd(v)$.
\end{proof}

We close with remarks on regularity and open questions. For $0<\nu<1$, the sharp power law gives a necessary condition for
H\"older regularity of the tangent. Indeed,
\[
 \gamma(\delta)\sim A^{\frac1\nu}\Lambda^{\frac\beta\nu}
                      \delta^{\frac\beta\nu}.
\]
If $\gamma\in C^{1,\tau}([0,\delta_0))$, with $0<\tau\le1$, then
$\gamma(0)=\gamma'(0)=0$ gives
$\gamma(\delta)=O(\delta^{1+\tau})$. Consequently
\[
 \tau\le\frac\beta\nu-1.
\]
This improves the necessary condition in Corollary~\ref{cor:noholder}.
The same restriction holds for any regular $C^{1,\tau}$
parametrization, by the change of parameter used in the proof of
Corollary~\ref{cor:boundarycurve}. This argument establishes only necessity.

The asymptotic equivalents determine the boundary gap and its leading
coefficient. They cannot be differentiated without further estimates.
In particular, a sharp equivalent for $b'(v)$ or $\gamma'(t)$ remains
a separate problem. Theorem~\ref{thm:endpointregularity} gives
continuity of the inverse tangent and $b'(v)\to-\infty$, but its
modulus need not be optimal. For $0<\nu<1$, we do not determine the
optimal H\"older regularity of the tangent. The power law for the
boundary position alone does not provide a two-point estimate for its
derivative. In this regime the coefficient in
\eqref{eq:amplitudeformula} depends on the global solution. A formula
in the model parameters alone, or a stable numerical method for this
coefficient, would provide additional information.
The finite-maturity problem introduces a time derivative and a moving
endpoint. The axis selection argument developed here does not directly
settle that problem.

\medskip
\noindent\textbf{Funding.} Takwon Kim is supported by the National
Research Foundation of Korea (NRF) grant funded by the Korea government
[Grant No. RS-2024-00351151]. Junkee Jeon is supported by the National
Research Foundation of Korea (NRF) grant funded by the Korea government
[Grant No. RS-2026-25591924]. Jehan Oh is supported by the National
Research Foundation of Korea (NRF) grant funded by the Korea government
[Grant Nos. RS-2025-00555316 and RS-2025-25415411].

\medskip
\noindent\textbf{Disclosure statement.} The authors report there are no
competing interests to declare.

\medskip
\noindent\textbf{Data availability.} No data sets were generated or
analyzed in this study. The code that produced Figure~\ref{fig:explicitcusp}
is available from the corresponding author upon reasonable request.

\medskip
\noindent\textbf{Declaration of generative AI use.} Large language model
tools were used to check the arguments and suggest corrections, to check
the bibliographic data of the cited works, and to correct the grammar of
the text. The authors verified every statement, proof and reference.
The authors assume responsibility for all content.

\appendix

\section{An explicit logarithmic contact curve}\label{sec:explicitcusp}

The following observation concerns the value function, rather than the
differentiability of the boundary $b$ or its inverse.

\begin{proposition}\label{prop:c1obstruction}
The function $w$ cannot have a $C^1$ extension to a neighborhood of
$(x_0,0)$ in the closed half-plane.
\end{proposition}

\begin{proof}
By continuity of $b$, interior contact points with positive variance
approach $(x_0,0)$. The gradient of $w$ is zero at these points. If the
asserted extension existed, continuity of its gradient would give
$\nabla w(x_0,0)=0$. Since $w(x_0,0)=0$, differentiability would imply
$W_\varepsilon\to0$ uniformly on each compact subset of the rescaled
closed half-plane.

Choose a nonnegative, nonzero smooth test function supported in
$\{X>0,Y>0\}$. On its support, \eqref{eq:scaledpde} holds for small
$\varepsilon$. Move all derivatives onto the test function and let
$\varepsilon\to0$. The left-hand side tends to zero, whereas the
right-hand side tends to $f_0$ times the integral of the test function.
This contradicts $f_0>0$.
\end{proof}

This does not rule out smooth fit in the stock variable alone, and it
does not decide the regularity of $\gamma$. The weighted estimates and
joint gradient continuity are distinct assertions. The following local
example makes this distinction explicit without using the global option
value. It also checks the logarithmic coefficient when $\nu=2$ and
$c=0$.

In a sufficiently small upper half-disc, use $0<\arg\zeta<\pi$ and set
\[
 F(\zeta)=C\zeta-\frac\zeta\pi\log\zeta,\qquad C>0.
\]
Choose the radius so that $\operatorname{Re}F'\ge m_*>0$.
The map extends continuously to the closed half-disc, with $F(0)=0$.
For distinct points $\zeta_1,\zeta_2$ in its closure, integration along
their segment gives
\[
 \operatorname{Re}\frac{F(\zeta_2)-F(\zeta_1)}{\zeta_2-\zeta_1}
 =\int_0^1\operatorname{Re}F'((1-t)\zeta_1+t\zeta_2) dt
 \ge m_*.
\]
On the real axis derivatives are upper traces. If the segment passes
through zero, the formula follows by an improper integral, since the
logarithmic singularity is integrable. Thus $F$ is injective on the
closure and is a homeomorphism onto its image there. Write $G=F^{-1}$.
The positive real axis maps to the positive real axis. The upper
trace of the negative axis has image
\begin{equation}\label{eq:exactcusp}
 x=-t(v),\quad y=v,\qquad
 t(v)=Cv+\frac v\pi\log\frac1v.
\end{equation}
The outer semicircle has image separated from zero. Thus in a small
half-disc about zero, the image is the region to the right of this
curve. Denote it by $\mathcal O_+$.

On $\mathcal O_+$ define
\[
 u(x,y)=\frac{y^2}{2}+\operatorname{Im}\left[
 \frac{G(x+iy)^2}{2}
 \left(C-\frac{\log G(x+iy)}\pi-\frac1{2\pi}\right)\right].
\]
Differentiating the primitive gives
\[
 u_x=\operatorname{Im}G,\qquad u_y=y+\operatorname{Re}G,
 \qquad \Delta u=1.
\]
On the curve \eqref{eq:exactcusp}, $G=-v$ and $y=v$, so
$u=0$ and $\nabla u=0$. Also $u=0$ on the positive axis.
Since $u_x>0$ inside, $u>0$ in $\mathcal O_+$. Extend $u=0$ to
the region on the other side of the curve. The bound
$|G'|=|1/F'|\le C_1$ shows that $u$ is locally $C^{1,1}$,
including across the free boundary. Consequently
$\Delta u=\chi_{\{u>0\}}$ in distributions.

Put $w_*=u/y$ for $y>0$. Then
\begin{equation}\label{eq:cuspoperator}
 \left[\frac y2(\partial_{xx}+\partial_{yy})+\partial_y\right]w_*
                         =\frac12\chi_{\{w_*>0\}}.
\end{equation}
Its boundary has exactly the coefficient $1/\pi$ in
\eqref{eq:sharpsuper} with $\sigma=1$, $\rho=c=0$, and $\nu=2$.

The functions $w_*$, $Dw_*$, and $yD^2w_*$ are locally bounded
up to the axis. To verify the first derivatives, write
$G=A+iB=re^{i\theta}$ and $L=C-\log r/\pi$. The imaginary part
of $F(G)$ gives $y=BL-A\theta/\pi$. If $A\ge0$, then
$A\theta\le B$. If $A<0$, the last term is nonnegative. Hence
$y\ge B(L-1/\pi)$ and $(w_*)_x=B/y$ is bounded.
Let $g(x)$ be the contact height, with $g(x)=0$ for $x\ge0$.
Since $u(x,g(x))=0$ and $g(x)u_y(x,g(x))=0$,
\[
 yu_y-u=\int_{g(x)}^y t u_{yy}(x,t) dt.
\]
Boundedness of $u_{yy}$ gives bounded $(w_*)_y$. Finally,
\[
 y(w_*)_{xx}=u_{xx},\quad
 y(w_*)_{xy}=u_{xy}-(w_*)_x,\quad
 y(w_*)_{yy}=u_{yy}-2(w_*)_y.
\]
The bounded gradient gives a Lipschitz extension of $w_*$ to the axis.

Nevertheless this gradient is not jointly continuous at zero.
At $z=iy$, the real part of $F(G(z))=iy$ gives
$A=-B\theta/(\pi L)=o(y)$, whereas $u_y(0,y)=y+A=y+o(y)$.
Thus
\[
 w_*(0,y)=\tfrac12y+o(y),\qquad
 (w_*)_y(0,y)=\frac{u_y(0,y)}y-\frac{u(0,y)}{y^2}
                         \longrightarrow\tfrac12.
\]
Interior contact points approach zero with $Dw_*=0$. A joint
$C^1$ extension would therefore have zero gradient at zero,
contradicting the displayed vertical quotient.

Figure~\ref{fig:explicitcusp} shows the contact geometry and the gradient
along two paths to the origin. The contact path is $(-2t(y),y)$, and it
lies strictly inside the contact set. Along the vertical path $(0,y)$,
the derivative $(w_*)_y$ instead tends to $1/2$.

\begin{figure}[t]
\centering
\begin{minipage}[t]{0.49\linewidth}
\centering
\begin{tikzpicture}
\begin{axis}[
 width=0.97\linewidth,height=5.8cm,
 xmin=-30,xmax=12,ymin=0,ymax=11,
 xlabel={$10^3x$},ylabel={$10^3y$},
 xtick={-30,-15,0,10},ytick={0,5,10},
 title={(a) Contact geometry},title style={font=\small},
 tick label style={font=\footnotesize},label style={font=\small},
 axis line style={black!65},tick style={black!65},
 clip=false]
\addplot[draw=none,fill=black!9] coordinates {
(-30,0) (0,0) (-7.51646128e-05,1.1e-05) (-0.000543069104,8.8e-05)
(-0.0017178626,0.000297) (-0.00387857044,0.000704) (-0.00728233911,0.001375) (-0.0121702102,0.002376)
(-0.018770446,0.003773) (-0.0273007045,0.005632) (-0.0379695777,0.008019) (-0.050977738,0.011)
(-0.0665188258,0.014641) (-0.0847801575,0.019008) (-0.105943304,0.024167) (-0.130184573,0.030184)
(-0.157675413,0.037125) (-0.188582763,0.045056) (-0.223069355,0.054043) (-0.261293977,0.064152)
(-0.303411703,0.075449) (-0.349574106,0.088) (-0.399929435,0.101871) (-0.454622787,0.117128)
(-0.513796254,0.133837) (-0.577589064,0.152064) (-0.646137701,0.171875) (-0.719576022,0.193336)
(-0.79803536,0.216513) (-0.881644626,0.241472) (-0.970530393,0.268279) (-1.06481698,0.297)
(-1.16462654,0.327701) (-1.27007912,0.360448) (-1.38129273,0.395307) (-1.49838341,0.432344)
(-1.62146529,0.471625) (-1.75065065,0.513216) (-1.88604997,0.557183) (-2.02777198,0.603592)
(-2.17592369,0.652509) (-2.33061046,0.704) (-2.49193605,0.758131) (-2.66000262,0.814968)
(-2.8349108,0.874577) (-3.01675973,0.937024) (-3.20564708,1.002375) (-3.40166907,1.070696)
(-3.60492054,1.142053) (-3.81549495,1.216512) (-4.03348442,1.294139) (-4.25897976,1.375)
(-4.49207049,1.459161) (-4.73284487,1.546688) (-4.98138992,1.637647) (-5.23779147,1.732104)
(-5.50213412,1.830125) (-5.77450134,1.931776) (-6.05497543,2.037123) (-6.34363758,2.146232)
(-6.64056788,2.259169) (-6.9458453,2.376) (-7.25954778,2.496791) (-7.58175218,2.621608)
(-7.91253435,2.750517) (-8.2519691,2.883584) (-8.60013027,3.020875) (-8.95709069,3.162456)
(-9.32292222,3.308393) (-9.69769578,3.458752) (-10.0814813,3.613599) (-10.474348,3.773)
(-10.8763637,3.937021) (-11.2875959,4.105728) (-11.7081109,4.279187) (-12.1379741,4.457464)
(-12.57725,4.640625) (-13.0260026,4.828736) (-13.4842947,5.021863) (-13.9521883,5.220072)
(-14.4297448,5.423429) (-14.9170246,5.632) (-15.4140873,5.845851) (-15.9209919,6.065048)
(-16.4377963,6.289657) (-16.964558,6.519744) (-17.5013336,6.755375) (-18.0481787,6.996616)
(-18.6051485,7.243533) (-19.1722974,7.496192) (-19.7496789,7.754659) (-20.337346,8.019)
(-20.9353508,8.289281) (-21.5437449,8.565568) (-22.1625789,8.847927) (-22.7919031,9.136424)
(-23.4317669,9.431125) (-24.0822191,9.732096) (-24.7433077,10.039403) (-25.4150801,10.353112)
(-26.0975833,10.673289) (-26.7908633,11) (-30,11)
} \closedcycle;
\addplot[hestonsuper,very thick,no marks] coordinates {
(0,0) (-7.51646128e-05,1.1e-05) (-0.000543069104,8.8e-05) (-0.0017178626,0.000297)
(-0.00387857044,0.000704) (-0.00728233911,0.001375) (-0.0121702102,0.002376) (-0.018770446,0.003773)
(-0.0273007045,0.005632) (-0.0379695777,0.008019) (-0.050977738,0.011) (-0.0665188258,0.014641)
(-0.0847801575,0.019008) (-0.105943304,0.024167) (-0.130184573,0.030184) (-0.157675413,0.037125)
(-0.188582763,0.045056) (-0.223069355,0.054043) (-0.261293977,0.064152) (-0.303411703,0.075449)
(-0.349574106,0.088) (-0.399929435,0.101871) (-0.454622787,0.117128) (-0.513796254,0.133837)
(-0.577589064,0.152064) (-0.646137701,0.171875) (-0.719576022,0.193336) (-0.79803536,0.216513)
(-0.881644626,0.241472) (-0.970530393,0.268279) (-1.06481698,0.297) (-1.16462654,0.327701)
(-1.27007912,0.360448) (-1.38129273,0.395307) (-1.49838341,0.432344) (-1.62146529,0.471625)
(-1.75065065,0.513216) (-1.88604997,0.557183) (-2.02777198,0.603592) (-2.17592369,0.652509)
(-2.33061046,0.704) (-2.49193605,0.758131) (-2.66000262,0.814968) (-2.8349108,0.874577)
(-3.01675973,0.937024) (-3.20564708,1.002375) (-3.40166907,1.070696) (-3.60492054,1.142053)
(-3.81549495,1.216512) (-4.03348442,1.294139) (-4.25897976,1.375) (-4.49207049,1.459161)
(-4.73284487,1.546688) (-4.98138992,1.637647) (-5.23779147,1.732104) (-5.50213412,1.830125)
(-5.77450134,1.931776) (-6.05497543,2.037123) (-6.34363758,2.146232) (-6.64056788,2.259169)
(-6.9458453,2.376) (-7.25954778,2.496791) (-7.58175218,2.621608) (-7.91253435,2.750517)
(-8.2519691,2.883584) (-8.60013027,3.020875) (-8.95709069,3.162456) (-9.32292222,3.308393)
(-9.69769578,3.458752) (-10.0814813,3.613599) (-10.474348,3.773) (-10.8763637,3.937021)
(-11.2875959,4.105728) (-11.7081109,4.279187) (-12.1379741,4.457464) (-12.57725,4.640625)
(-13.0260026,4.828736) (-13.4842947,5.021863) (-13.9521883,5.220072) (-14.4297448,5.423429)
(-14.9170246,5.632) (-15.4140873,5.845851) (-15.9209919,6.065048) (-16.4377963,6.289657)
(-16.964558,6.519744) (-17.5013336,6.755375) (-18.0481787,6.996616) (-18.6051485,7.243533)
(-19.1722974,7.496192) (-19.7496789,7.754659) (-20.337346,8.019) (-20.9353508,8.289281)
(-21.5437449,8.565568) (-22.1625789,8.847927) (-22.7919031,9.136424) (-23.4317669,9.431125)
(-24.0822191,9.732096) (-24.7433077,10.039403) (-25.4150801,10.353112) (-26.0975833,10.673289)
(-26.7908633,11)
};
\addplot[black!65,densely dashed,thick,->,no marks] coordinates {
(-22.0602848,4) (-21.6144584,3.90875) (-21.1672757,3.8175) (-20.7187043,3.72625)
(-20.2687103,3.635) (-19.8172577,3.54375) (-19.3643092,3.4525) (-18.9098252,3.36125)
(-18.4537639,3.27) (-17.9960813,3.17875) (-17.5367309,3.0875) (-17.0756634,2.99625)
(-16.6128264,2.905) (-16.1481645,2.81375) (-15.6816182,2.7225) (-15.2131246,2.63125)
(-14.742616,2.54) (-14.2700199,2.44875) (-13.7952587,2.3575) (-13.3182484,2.26625)
(-12.8388984,2.175) (-12.3571105,2.08375) (-11.8727779,1.9925) (-11.385784,1.90125)
(-10.8960009,1.81) (-10.4032879,1.71875) (-9.90748934,1.6275) (-9.40843202,1.53625)
(-8.90592215,1.445) (-8.39974142,1.35375) (-7.88964204,1.2625) (-7.3753403,1.17125)
(-6.85650817,1.08) (-6.33276197,0.98875) (-5.80364698,0.8975) (-5.26861552,0.80625)
(-4.72699525,0.715) (-4.17794094,0.62375) (-3.62035769,0.5325) (-3.05277049,0.44125)
(-2.47308289,0.35)
};
\draw[hestoncritical,dashed,thick,->]
 (axis cs:0,9) -- (axis cs:0,1.2);
\node[font=\scriptsize] at (axis cs:-25.5,6.8) {contact};
\node[font=\scriptsize] at (axis cs:-5,10) {continuation};
\node[font=\scriptsize,anchor=west,text=hestoncritical]
 at (axis cs:1,5.5) {$(0,y)$};
\node[font=\scriptsize,anchor=west,text=black!70]
 at (axis cs:-27,1.9) {$(-2t(y),y)$};
\fill (axis cs:0,0) circle[radius=1.4pt];
\draw[black,thick,->] (axis cs:0,0) -- (axis cs:10,0);
\node[font=\scriptsize,anchor=south east]
 at (axis cs:11,0.3) {tangent};
\end{axis}
\end{tikzpicture}
\end{minipage}\hfill
\begin{minipage}[t]{0.49\linewidth}
\centering
\begin{tikzpicture}
\begin{axis}[
 width=0.97\linewidth,height=5.8cm,
 xmin=0,xmax=80,ymin=-0.025,ymax=0.56,
 xlabel={$-\log_{10}y$},ylabel={First derivatives},
 xtick={0,20,40,60,80},ytick={0,0.25,0.5},
 title={(b) Two approaches},title style={font=\small},
 tick label style={font=\footnotesize},label style={font=\small},
 axis line style={black!65},tick style={black!65},
 grid=major,grid style={black!10},
 legend style={at={(0.98,0.68)},anchor=north east,
  draw=none,fill=white,font=\scriptsize,cells={anchor=west}},
 no marks]
\addplot[black!40,densely dotted,forget plot] coordinates {(0,0.5) (80,0.5)};
\addplot[hestonsuper,very thick] coordinates {
(2,0.343266105) (2.25,0.321696525) (2.5,0.302672328) (2.75,0.285780479)
(3,0.270688199) (3.25,0.257126038) (3.5,0.24487436) (3.75,0.233752826)
(4,0.223612249) (4.25,0.214328311) (4.5,0.205796676) (4.75,0.197929197)
(5,0.190650923) (5.25,0.183897741) (5.5,0.177614504) (5.75,0.171753519)
(6,0.166273341) (6.25,0.16113778) (6.5,0.156315098) (6.75,0.151777346)
(7,0.147499817) (7.25,0.143460585) (7.5,0.139640131) (7.75,0.136021014)
(8,0.132587608) (8.25,0.129325869) (8.5,0.126223137) (8.75,0.123267975)
(9,0.120450019) (9.25,0.117759858) (9.5,0.115188926) (9.75,0.112729409)
(10,0.110374165) (11,0.101872819) (12,0.0946015418) (13,0.0883098591)
(14,0.0828112032) (15,0.0779637125) (16,0.0736575701) (17,0.0698064196)
(18,0.0663414044) (19,0.0632069393) (20,0.0603576527) (21,0.0577561392)
(22,0.0553712806) (23,0.0531769759) (24,0.0511511668) (25,0.0492750829)
(26,0.0475326495) (27,0.0459100201) (28,0.0443952018) (29,0.042977755)
(30,0.0416485494) (31,0.0403995643) (32,0.0392237249) (33,0.0381147665)
(34,0.0370671216) (35,0.0360758256) (36,0.0351364368) (37,0.0342449695)
(38,0.0333978367) (39,0.0325918015) (40,0.0318239351) (41,0.0310915815)
(42,0.0303923256) (43,0.0297239672) (44,0.029084497) (45,0.0284720768)
(46,0.0278850212) (47,0.0273217823) (48,0.0267809361) (49,0.0262611697)
(50,0.0257612713) (51,0.0252801203) (52,0.0248166791) (53,0.0243699853)
(54,0.0239391451) (55,0.0235233273) (56,0.023121758) (57,0.0227337155)
(58,0.0223585262) (59,0.0219955605) (60,0.0216442294) (61,0.0213039811)
(62,0.0209742984) (63,0.0206546956) (64,0.0203447164) (65,0.0200439319)
(66,0.0197519381) (67,0.0194683544) (68,0.0191928219) (69,0.018925002)
(70,0.0186645748) (71,0.0184112376) (72,0.0181647044) (73,0.0179247042)
(74,0.0176909803) (75,0.0174632892) (76,0.0172413999) (77,0.0170250931)
(78,0.0168141604) (79,0.0166084036) (80,0.0164076342)
};
\addlegendentry{$(w_*)_x(0,y)$}
\addplot[hestoncritical,very thick,dashed] coordinates {
(2,0.461705072) (2.25,0.46704097) (2.5,0.471323231) (2.75,0.474812672)
(3,0.477694281) (3.25,0.480102104) (3.5,0.482135145) (3.75,0.483867794)
(4,0.485356819) (4.25,0.486646146) (4.5,0.487770197) (4.75,0.48875626)
(5,0.489626192) (5.25,0.490397671) (5.5,0.491085118) (5.75,0.4917004)
(6,0.492253348) (6.25,0.492752167) (6.5,0.493203751) (6.75,0.493613923)
(7,0.493987632) (7.25,0.494329107) (7.5,0.494641977) (7.75,0.494929371)
(8,0.495194003) (8.25,0.495438229) (8.5,0.49566411) (8.75,0.49587345)
(9,0.496067837) (9.25,0.49624867) (9.5,0.496417189) (9.75,0.496574494)
(10,0.496721566) (11,0.497223689) (12,0.497617911) (13,0.497933205)
(14,0.498189406) (15,0.49840047) (16,0.498576451) (17,0.498724745)
(18,0.498850892) (19,0.498959107) (20,0.499052647) (21,0.49913406)
(22,0.499205362) (23,0.499268166) (24,0.499323774) (25,0.499373249)
(26,0.499417464) (27,0.499457141) (28,0.499492881) (29,0.499525191)
(30,0.499554497) (31,0.499581161) (32,0.499605492) (33,0.499627756)
(34,0.49964818) (35,0.499666963) (36,0.499684275) (37,0.499700268)
(38,0.499715072) (39,0.499728802) (40,0.49974156) (41,0.499753436)
(42,0.499764509) (43,0.499774851) (44,0.499784524) (45,0.499793585)
(46,0.499802085) (47,0.499810069) (48,0.499817579) (49,0.49982465)
(50,0.499831317) (51,0.49983761) (52,0.499843557) (53,0.499849182)
(54,0.499854508) (55,0.499859556) (56,0.499864346) (57,0.499868894)
(58,0.499873216) (59,0.499877328) (60,0.499881243) (61,0.499884972)
(62,0.499888529) (63,0.499891922) (64,0.499895163) (65,0.49989826)
(66,0.499901221) (67,0.499904055) (68,0.499906768) (69,0.499909368)
(70,0.49991186) (71,0.49991425) (72,0.499916544) (73,0.499918748)
(74,0.499920865) (75,0.4999229) (76,0.499924858) (77,0.499926741)
(78,0.499928555) (79,0.499930302) (80,0.499931986)
};
\addlegendentry{$(w_*)_y(0,y)$}
\addplot[black!65,thick,dashdotted] coordinates {(2,0) (80,0)};
\addlegendentry{contact: both zero}
\end{axis}
\end{tikzpicture}
\end{minipage}
\caption{The explicit local example with $C=1$.
(a) The exact curve $x=-t(y)$ from \eqref{eq:exactcusp}, with the
contact set shaded. Both marked paths approach the origin, whose
boundary tangent is horizontal. The axes have different display scales.
(b) First derivatives evaluated from the conformal inverse $G$ along
$(0,y)$, compared with their zero values on $(-2t(y),y)$.
The dotted line marks $1/2$. This is a local example with a prescribed
source, not a numerical solution of the perpetual put problem.}
\label{fig:explicitcusp}
\end{figure}
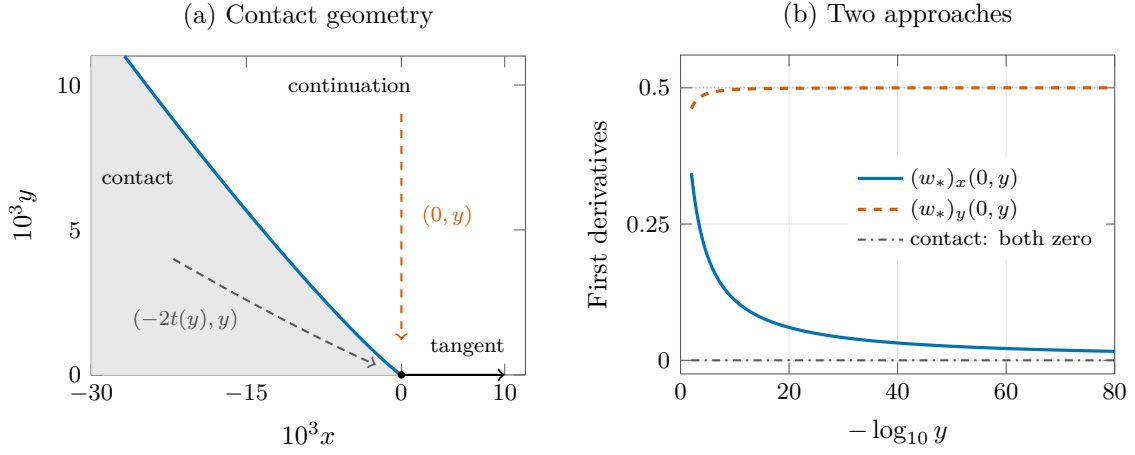

We now compare the example with the weighted regularity theorem
of~\cite{DF}. The example can be written using the full Heston operator.
Choose $\sigma=1$, $\rho=0$, $\kappa=\theta=1$, and $r=q>0$.
Then $\Lop=\frac y2\Delta+\partial_y+T$, where
$T=-\frac y2\partial_x-y\partial_y-r$.
For $A_{\rm op}=-\Lop$, obstacle zero, and source
$f_*=-\frac12-Tw_*$, equation \eqref{eq:cuspoperator} gives
\[
 \min\{A_{\rm op}w_*-f_*,w_*\}=0.
\]
Here the equality holds almost everywhere: its first entry is zero in
continuation and equals $1/2$ in the contact set.

Take $V=B_{R_0}(0)\cap\HH$ with $R_0$ small enough that the construction
holds on its closure. We check the hypotheses of
\cite[Theorem 1.1]{DF} on this half-disc. The obstacle is zero, hence
belongs to $C_s^{2+\lambda}(\overline V)$ for every $0<\lambda<1$,
and the interest rate is positive. The weighted bounds above make
$Tw_*$ Lipschitz: the weak derivatives of $y(w_*)_x$ and $y(w_*)_y$
are bounded, and these two functions extend by zero on the axis.
Thus $f_*$ is Lipschitz on $\overline V$. If $s(z,z')$ is the cycloidal
distance of \cite[(2.1)]{DF} and
$\ell=|x-x'|+|y-y'|$, boundedness of $V$ gives
\[
 \ell=s(z,z')\bigl(\sqrt y+\sqrt{y'}+\sqrt\ell\bigr)
       \le C_Vs(z,z').
\]
It follows that $f_*\in C_s^\lambda(\overline V)$ for every
$0<\lambda<1$.

For these parameters, the weight in \cite[Definition 2.1]{DF} is
$ye^{-\gamma_0|x|-2y}$ for a fixed $\gamma_0>0$.
It is comparable to $y$ on $V$. Since $w_*$, $Dw_*$ and $yD^2w_*$
are bounded, the weighted Sobolev norm of $w_*$ satisfies
\begin{align*}
 &
 \int_V\left(y^2|D^2w_*|^2+(1+y)^2|Dw_*|^2
          +(1+y)|w_*|^2\right)ye^{-\gamma_0|x|-2y} dx dy\\
 &\qquad\le C\int_V y dx dy<\infty.
\end{align*}
The weak derivatives here have no free-boundary measure, because $w_*$
and its gradient match zero across the contact curve for $y>0$.
Hence $w_*$ belongs to the weighted space $H^2(V,\cdot)$ of
\cite[Definition 2.1]{DF} and to $C(\overline V)$, as required.

However, $w_*$ is not $C^1$ at the origin, which lies in the relative
interior of the degenerate boundary of $V$. The estimate in
\cite[Section 5, (5.2)]{DF} bounds $|w_*|$, $|Dw_*|$ and
$y|D^2w_*|$ at positive variance, and all these bounds hold here. They
do not imply continuity of $Dw_*$ at the axis. The space
$C_s^{1,1}(\overline U)$ of \cite[Definition 2.2]{DF}, with
$U=B_{R_0/2}(0)\cap\HH$, also contains the requirement
$C^1(\overline U)$, and this requirement fails for $w_*$. This
comparison concerns a local obstacle problem with a prescribed
Lipschitz source. It does not identify $w_*$ with the global put value,
and it is consistent with the weighted derivative estimate itself.

\end{document}